\documentclass[12pt, leqno, a4paper]{amsart}
\usepackage{amsmath,mathrsfs}
\usepackage{amssymb}
\usepackage{color}
\usepackage{calligra}
\usepackage{dsfont}
\usepackage{autobreak}
\usepackage{pstricks,pst-node}
\usepackage[latin1]{inputenc}
\usepackage{todonotes}

\allowdisplaybreaks
\usepackage{pdfsync}

\newcommand\NoBlackBoxes{\global\overfullrule0pt}
\NoBlackBoxes

\usepackage{enumitem}
\setitemize{noitemsep,topsep=0pt,parsep=0pt,partopsep=0pt}

\newcommand{\eps}{\varepsilon}

\newcommand{\N}{\mathbb{N}}

\renewcommand{\P}{\mathbb{P}}

\newcommand{\eee}{{\rm e}}

\makeatletter
\let\serieslogo@\relax
\let\@setcopyright\relax

\theoremstyle{plain}
\newtheorem{theorem}{Theorem}[section]
\newtheorem{lemma}[theorem]{Lemma}
\newtheorem{corollary}[theorem]{Corollary}
\newtheorem{proposition}[theorem]{Proposition}

\theoremstyle{definition}

\theoremstyle{remark}
\newtheorem{remark}[theorem]{Remark}
\newtheorem{rem}[theorem]{Remark}

\renewcommand{\P}{{\mathbb{P}}}
\newcommand{\E}{{\mathbb{E}}}
\newcommand{\R}{{\mathbb{R}}}

\newcommand{\C}{\mathbb{C}}

\newcommand{\V}{\mathbb{V}}

\newcommand{\ind}{\operatorname{1}}

\renewcommand{\epsilon}{\varepsilon}
\renewcommand{\phi}{\varphi}

\numberwithin{equation}{section}

\begin{document}

\setcounter{page}{1}

\title[Fluctuations for Ising models on Erd\H{o}s-R\'enyi graphs]{Fluctuations of the Magnetization for Ising models
on Erd\H{o}s-R\'enyi random graphs -- the regimes of low temperature and external magnetic field}

\author[Zakhar Kabluchko]{Zakhar Kabluchko}
\address[Zakhar Kabluchko]{Fachbereich Mathematik und Informatik,
Universit\"at M\"unster,
Einsteinstra\ss e 62,
48149 M\"unster,
Germany}

\email[Zakhar Kabluchko]{zakhar.kabluchko@uni-muenster.de}

\author[Matthias L\"owe]{Matthias L\"owe}
\address[Matthias L\"owe]{Fachbereich Mathematik und Informatik,
Universit\"at M\"unster,
Einsteinstra\ss e 62,
48149 M\"unster,
Germany}

\email[Matthias L\"owe]{maloewe@math.uni-muenster.de}

\author[Kristina Schubert]{Kristina Schubert}
\address[Kristina Schubert]{ Fakult\"at f\"ur Mathematik, TU Dortmund, Vogelpothsweg 87, 44227 Dortmund,
Germany}

\email[Kristina Schubert]{kristina.schubert@tu-dortmund.de}


\date{\today}

\subjclass[2000]{Primary: 82B44; Secondary: 82B20}

\keywords{Ising model, dilute Curie-Weiss model, fluctuations, Central Limit Theorem, random graphs}

\newcommand{\wlim}{\mathop{\hbox{\rm w-lim}}}
\newcommand{\na}{{\mathbb N}}
\newcommand{\re}{{\mathbb R}}

\newcommand{\vep}{\varepsilon}

\begin{abstract}
We continue our analysis of Ising models on the (directed) Erd\H{o}s-R\'enyi random graph $G(N,p)$.  We
prove a quenched Central Limit Theorem for the magnetization and describe the fluctuations of the log-partition function. In the current note
we consider  the low temperature regime $\beta>1$ and the case when an external magnetic field is present.
In both cases, we assume that $p=p(N)$ satisfies $p^3N \to \infty$.
\end{abstract}

\maketitle

\section{Introduction and main results}
\subsection{Description of the model}
In this paper we continue our investigation of Ising models on the Erd\H{o}s-R\'enyi random graph.
Technically speaking they are disordered ferromagnets in the sense of Fr\"ohlich's lecture~\cite{froehlichlecture}.
The model we are studying was introduced and first analyzed by Bovier and Gayrard in \cite{BG93b}.
The topology of this model is given by a directed Erd\H{o}s-R\'enyi graph $G=G(N,p)$. Its vertex set is the set $\{1,
\ldots, N\}$, and each of the directed edges $(i,j)$ is realized with probability $p\in (0,1]$ independently of all other edges. For simplicity, the
case $i=j$ is admitted. The random variables $\vep_{i,j}^N=\vep_{i,j}, i,j \in \{1, \ldots, N\}$, which indicate whether an edge
$(i,j)$ is present or not, are thus assumed to be i.i.d.~with
distribution
\begin{equation*}
\P[\vep_{i,j} = 1] = p, \qquad \P[\vep_{i,j} = 0] = 1- p.
\end{equation*}
Our general assumption is that $p=p(N)$ satisfies
$p^3N \to \infty$ as $N\to\infty$.
This is more than enough to ensure that asymptotically almost surely the graph is connected.
It is likely that one could prove variants of our central results also under the weaker assumption $pN \to \infty$, but our main technique runs into problems.

On a fixed realization of this Erd\H{o}s-R\'enyi random graph $G$ we define the Hamiltonian or energy function of the Ising model. It is given by the function $$H= H_N: \{-1,+1\}^N \to\R$$ defined as
\begin{equation}\label{hamilCW}
H(\sigma) := - \frac 1 {2Np} \sum_{i,j=1}^N \vep_{i,j} \sigma_i \sigma_j-h \sum_{i=1}^N \sigma_i
\end{equation}
for $\sigma = (\sigma_1,\ldots,\sigma_N) \in \{-1,+1\}^N$. Here $h>0$ plays the role of an external magnetic field.
With such an energy function $H$ we associate a Gibbs measure on $\{-1,+1\}^N$. This is a random probability measure
with respect to the randomness encoded by the $(\vep_{i,j})_{i,j=1}^N$. It is given by
\begin{equation}\label{gibbs}
\mu_\beta (\sigma):=\frac 1 {Z_{N}(\beta)} \exp(-\beta H(\sigma)), \qquad \sigma \in \{-1,+1\}^N,
\end{equation}
where $\beta > 0$ is called the inverse temperature. The normalizing constant is the partition function given by
\begin{equation*}
Z_N(\beta):= \sum_{\sigma \in \{-1,+1\}^N} \exp(-\beta H(\sigma)).
\end{equation*}

The well studied Curie-Weiss model is a special case of this setup, namely the situation where $p\equiv 1$.
For a survey over many results on this model
see  \cite{Ellis-EntropyLargeDeviationsAndStatisticalMechanics}. The Curie-Weiss model is an
interesting model for ferromagnetism because it exhibits a phase transition at the critical temperature  $\beta_c=1$.
This phase transition can  be seen by analyzing
the  magnetization per particle
\begin{equation*}
m_N(\sigma):= \frac{\sum_{i=1}^N \sigma_i}{N}= \frac{|\sigma|}{N}.
\end{equation*}
Here we have introduced the notation
\begin{equation*}
|\sigma|:=\sum_{i=1}^N \sigma_i= N m_N(\sigma).
\end{equation*}
To avoid possible confusion, observe that $|\sigma|$ can be negative. In the Curie-Weiss model with $h=0$ the distribution of the magnetization per particle $m_N$ under the Gibbs measure converges
to
\begin{equation*}
\frac 12 (\delta_{m^+(\beta)}+\delta_{m^-(\beta)}).
\end{equation*}
Here $\delta_x$ is the Dirac-measure in a point $x$, while $m^+(\beta)$ is the largest  solution of
\begin{equation}\label{eq:CWeq}
z= \tanh(\beta z),
\end{equation}
and $m^-(\beta)=-m^+(\beta)$.
Observe that for $\beta\le 1$ the above equation \eqref{eq:CWeq} has only the trivial solution $m^+(\beta)=0$. Therefore in
the high temperature regime $\beta \le 1$ the magnetization per particle $m_N$ converges to $0$ in probability. For
$\beta>1$ the largest solution of \eqref{eq:CWeq} is strictly positive. Hence in the low temperature regime $\beta >
1$ the magnetization $m_N$ is asymptotically concentrated in two values, a positive one and a negative one.

Let us now turn to the Curie-Weiss model with $h >0$ and  $\beta >0$. Then, it is known~\cite{Ellis-EntropyLargeDeviationsAndStatisticalMechanics} that the distribution
of the magnetization per particle $m_N$ under the Gibbs measure
converges to
$\delta_{m^+(\beta,h)}$. Here $m^+(\beta,h)$ is the largest solution of
\begin{equation}\label{eq:CWeq_mit_h}
z= \tanh(\beta (z+h)).
\end{equation}

In \cite{BG93b} the authors show that the same limit theorems for $m_N$ hold in the
dilute Curie-Weiss Ising model, that we defined in \eqref{hamilCW} and \eqref{gibbs}, as long as $p N \to \infty$.
More concretely, if $\beta \le 1$ and $h=0$, then for
almost all realizations of the random graph, the magnetization $m_N$ converges to $0$ in probability under the Gibbs measure, while for $\beta >1$ and $h=0$ it converges to $\frac 12
(\delta_{m^+(\beta)}+\delta_{m^-(\beta)})$ in distribution. Moreover, for all $\beta>0$ and $h>0$,  $m_N$ again
converges in probability to $\delta_{m^+(\beta,h)}$. This latter fact is not explicitly stated in \cite{BG93b}, however, it can be
easily derived with the methods introduced there.

The starting point of our investigations in \cite{KaLS19a}, \cite{KaLS19c}, and \cite{KaLS19b}, as well as the
current note, is the observation that in the Curie-Weiss model one can also prove a Central Limit Theorem for the
magnetization when $\beta <1$ and $h=0$
(see, e.g.~\cite{Chatterjee_Shao}, \cite{EL10}, \cite{Ellis-EntropyLargeDeviationsAndStatisticalMechanics}, \cite{Ellis_Newman_78b}). These references show that $\sqrt N m_N$ converges in distribution to a normal random
variable with mean $0$ and variance $\frac 1 {1-\beta}$.
Furthermore, as can be expected from this result, at $\beta=1$, there is no such standard Central Limit Theorem and
one has to scale in a different way. The result is that $\sqrt[4] N m_N$ converges in distribution to a non-normal
random variable with density proportional to $\exp(-\frac 1 {12} x^4)$ with respect to the Lebesgue measure.
If $\beta >1$, one has to consider the conditional distribution of $\sqrt N (m_N-m^+(\beta))$ conditioned to $m_N$
being positive. In this case $\sqrt N (m_N-m^+(\beta))$ conditioned on $m_N >0$ converges in distribution to a
normal distribution with expectation $0$ and variance
\begin{equation}\label{eq:var_LT}
\sigma^2(\beta,0):=\sigma^2(\beta):= \frac{1-m^+(\beta)^2}{1-\beta(1-m^+(\beta)^2)}.
\end{equation}
Similarly, when $h>0$ the random variable
$\sqrt N (m_N-m^+(\beta,h))$ converges in distribution to a normal distribution with expectation $0$ and variance
\begin{equation}\label{eq:var_ext}
\sigma^2(\beta,h):= \frac{1-m^+(\beta,h)^2}{1-\beta(1-m^+(\beta,h)^2)}.
\end{equation}
These results can be found in \cite{EllisNewman_80}.

The general question we have been investigating in previous articles was, whether such limit theorems also hold in
the dilute setting introduced above.
To state our results let us introduce
the following {\it random} element of the space of finite measures on $\R$, denoted by $\mathcal M(\R)$:
\begin{equation}\label{eq:def_L_N}
L_N := \frac 1{Z_N(\beta)} \sum_{\sigma\in \{-1,+1\}^N} \eee^{-\beta H(\sigma)}  \delta_{\frac 1 {\sqrt N}
(\sum_{i=1}^N \sigma_i-Nm)},
\end{equation}
where $m$ is either $m^+(\beta)$ (resp. $m^-(\beta)$) or $m^+(\beta, h)$ depending on the case we consider).
Then the probability measure $L_N$ is random, since it depends on the random variables $\vep_{i,j}, i,j \in \{1, \ldots, N\}$.
In \cite{KaLS19a} we showed that, if $p^3 N^2 \to \infty$, $\beta <1$, and $h=0$ the random element $L_N$ converges
in probability to the normal distribution with mean 0 and variance $\frac{1}{1-\beta}$, denoted by $\mathfrak N_{0,
1/(1-\beta)}$, which is an element in $\mathcal M(\R)$.
In \cite{KaLS19b} we extended this result to the situation of $p N \to \infty$, $\beta <1$, and $h=0$. Moreover, for
$p^4 N^3 \to \infty$, $\beta =1$, and $h=0$
we considered
\begin{equation*}
L_N^1 := \frac 1{Z_N(\beta)} \sum_{\sigma\in \{-1,+1\}^N} \eee^{-\beta H(\sigma)}  \delta_{\frac 1 {N^{3/4}}
\sum_{i=1}^N \sigma_i}
\end{equation*}
and showed that it converges in probability to $\mathfrak{M} \in \mathcal M(\R)$, the probability measure with density
\begin{equation*}
\psi(x):=\frac{\eee^{-\frac 1 {12} x^4}}{\int_\R \eee^{-\frac 1 {12} y^4} dy}.
\end{equation*}
For smaller values of $p$, $\beta=1$, and $h=0$ we showed that suitable versions of $L_N^1$ again have a normal
distribution as limiting distribution (again in probability). Finally, in \cite{KaLS19c} we analyzed the fluctuations
of the random partition function $Z_N(\beta)$.

Note that the situation we analyzed in \cite{KaLS19a}, \cite{KaLS19c}, and\cite{KaLS19b}, as well as in the present
note is of a different character than the results for sparse (Erd\H{o}s-R\'enyi) graphs. Such situations were deeply
studied by Dembo and Montanari in \cite{Dembo_Montanari_2010b} and  \cite{Dembo_Montanari_2010a}and by van der
Hofstad and coauthors in
\cite{van_der_Hofstad_et_al_2015b},
\cite{van_der_Hofstad_et_al_2010},
\cite{van_der_Hofstad_et_al_2014},
\cite{van_der_Hofstad_et_al_2015c}, and
\cite{van_der_Hofstad_et_al_2015}.

\subsection{Main results}
As announced, in this note we will study the fluctuations of $m_N$, when either $h=0$ and $\beta>1$, or when $h>0$ and $\beta>0$ is arbitrary. In both cases we require that
$p$ is such that $p^3N \to \infty$.
The results will be formulated in terms of the quantity $L_N$ defined in \eqref{eq:def_L_N} and related quantities $L_N^+$ and $L_N^-$ to be defined below.
Recall that $L_N$ is a random element of $\mathcal M(\R)$, the set of finite measures on
$\R$
and that $\mathcal M (\R)$. We endow $\mathcal M(\R)$ with the
topology of weak convergence and denote by $\rho_{\text{weak}}$ any metric  generating the weak topology and turning $\mathcal M (\R)$ into a complete separable metric space.

Our main results are Central Limit Theorems for $m_N$. The first one is
\begin{theorem}\label{theo:b>1}
Assume that $h=0$, $\beta> 1$, and $p^3 N \to\infty$ as $N\to\infty$.
Then,
$$
L_N^+ := \frac {2} {Z_N(\beta)} \sum_{\sigma\in \{-1,+1\}^N} \eee^{-\beta H(\sigma)}  \delta_{\frac 1 {\sqrt N}
(\sum_{i=1}^N \sigma_i-Nm^+(\beta))} \ind_{\sum_{i=1}^N \sigma_i>0},
$$
considered as a random element of $\mathcal M(\R)$,
converges in probability to the normal distribution $\mathfrak N_{0,
\sigma^2(\beta)}$, which is considered as a deterministic element of $\mathcal M(\R)$.
Here the variance $\sigma^2(\beta)$ is given by \eqref{eq:var_LT}.
In other words, for every $\eps>0$,
$$
\lim_{N\to\infty} \P\left[\rho_{\text{weak}}\left(L_N^+, \mathfrak N_{0,\sigma^2(\beta)}\right)>\eps\right] = 0.
$$

An analogous assertion holds true, if, in the definition of $L_N^+$,
we replace $m^+(\beta)$ by $-m^+(\beta)=m^-(\beta)$ and
restrict our attention to configurations with negative magnetization, i.e.\ if we consider
$$
L_N^- := \frac 2{Z_N(\beta)} \sum_{\sigma\in \{-1,+1\}^N} \eee^{-\beta H(\sigma)}  \delta_{\frac 1 {\sqrt N}
(\sum_{i=1}^N \sigma_i-Nm^-(\beta))} \ind_{\sum_{i=1}^N \sigma_i\leq 0}.
$$
\end{theorem}
\begin{remark}
Intuitively, approximately one half of the configurations are such that the magnetization per particle is
close to $m^+(\beta)$, whereas for the other half it is close to $m^-(\beta)$.
This is very different from the high-temperature setting considered in the previous publications \cite{KaLS19a} and \cite{KaLS19b}, where the
magnetization per particle was concentrated near $0$.
In $L_N^+$ we only  take the
configurations with positive magnetization into account, which is why a factor of $2$ is necessary in its definition. Without the factor, the limit would be a measure of total mass $1/2$.
\end{remark}
Our second main theorem is
\begin{theorem}\label{theo:clt_magn_field}
Assume  that $h>0,\beta >0$, and $p^3 N \to\infty$ as $N\to\infty$.
Then,
$$
L_N:= \frac {1} {Z_N(\beta)} \sum_{\sigma\in \{-1,+1\}^N} \eee^{-\beta H(\sigma)}  \delta_{\frac 1 {\sqrt N}
(\sum_{i=1}^N \sigma_i-Nm^+(\beta,h))},
$$
considered as a random element of $\mathcal M(\R)$, converges in probability to $\mathfrak N_{0,
\sigma^2(\beta,h)}$. Here $m^+(\beta,h)$ and $\sigma^2(\beta,h)$ are given by~\eqref{eq:CWeq_mit_h} and~\eqref{eq:var_ext}.
In other words, for every $\eps>0$,
$$
\lim_{N\to\infty} \P[\rho_{\text{weak}}(L_N, \mathfrak N_{0,\sigma^2(\beta,h)})>\eps] = 0.
$$
\end{theorem}

In \cite{KaLS19c} we also analyzed the fluctuations of $Z_N(\beta, h)$ for $\beta <1$, $h=0$. Before we state  the corresponding result for $h=0$, $\beta>1$ and for $h>0, \beta>0$ we introduce some notation.
We set
\begin{equation*}
m := \begin{cases} m^+(\beta)& \text{if } h=0, \beta>1  \\ m^+(\beta, h) & \text{if } h>0, \beta>0 \end{cases}
\end{equation*}
and
\begin{equation}\label{tildeZN(g)_2}
\tilde Z_N(\beta):= \sum_{\sigma \in \{-1, +1\}^N} \eee^{-\beta H(\sigma)}\exp\left( -\frac{\beta}{2 Np} m^2 \sum_{i,j=1}^N \vep_{i,j} \right).
\end{equation}
We comment on the role of $\tilde Z_N(\beta)$ at the beginning of Section \ref{sec_proof_th_1-1}. In the following theorem, it is possible to replace $\E \tilde Z_N(\beta)$ by its asymptotic forms given in~\eqref{eq:E_tilde_Z_N} below  and~\eqref{eq:E_tilde_asympt_magnetic}, but we prefer not to state these long expressions explicitly.
\begin{theorem}\label{theo:fluc_part}
Assume  that $h=0$, $\beta >1$ or $h>0$, $\beta>0$. Further assume  that $p=p(N)$ is such that $p^3 N \to\infty$ as $N\to\infty$ and, moreover,
that $p(N)$ is bounded away from $1$.
Then,
\begin{equation*}
\frac{\log (Z_N(\beta)/\E \tilde Z_N(\beta) )-\frac{\beta N m^2}2}{\sqrt{\frac{\beta^2m^4(1-p)}{4p}}} \to \mathfrak N_{0,1}
\end{equation*}
in distribution.

For fixed $0<p<1$ that does not depend on $N$ this may be rewritten as
\begin{equation*}
\frac{Z_N(\beta)/\E \tilde Z_N(\beta)}{\eee^{\frac{\beta N m^2}2}}\to \eee^{\zeta}
\end{equation*}
in distribution. Here $\zeta$ denotes a normal random variable with expectation $0$ and variance $\frac{\beta^2m^4(1-p)}{4p}$.
\end{theorem}

We will prepare the proof of these theorems analytically in the following section. The actual proofs will follow in
Sections 3--5.

\begin{rem}
While we were working on this manuscript Deb and Mukherjee \cite{deb_mukherjee} published some really interesting results on the fluctuations of the magnetization of Ising models on general almost regular graphs. Their results partially confirm our results in 
\cite{KaLS19a} and \cite{KaLS19b}, as well as those of Theorems \ref{theo:b>1} and \ref{theo:clt_magn_field}. However, their techniques are completely different from ours. On the other hand, our third main result, Theorem \ref{theo:fluc_part}, is new. Its proof relies on the results developed in order to prove Theorems \ref{theo:b>1} and \ref{theo:clt_magn_field}.  
\end{rem}

\section{Technical preparation}
In the proof of our main theorems we will encounter some functions for which we will need an expansion up to certain
orders. These functions will be studied and analyzed in this section.

More precisely, for arbitrary complex variables $z$ and $p$ let
us define the function
\begin{equation}\label{eq:centralfunc}
F(p, z) := \log (1 - p + p\eee^{z}).
\end{equation}
Note that in this section we do not assume that $p$ is a probability.
We will next compute the power series expansion of some linear combinations of $F(p,z)$
in $p$ and $z$ variables around the origin $(0,0)$.

Note that, for $|p|< 2$ and $|z|< z_0$ with sufficiently small $z_0>0$, we can estimate that
$$
|p(\eee^z-1)| < 1.
$$
Thus, the function $F(p, z)$ is an analytic function of two complex variables $p$ and $z$ on the domain
$$
\mathcal D = \{(p,z)\in\C^2\colon |p| < 2, |z|<z_0\}.
$$
Therefore, it has a power series expansion  which converges uniformly and absolutely on compact subsets of this
domain. In particular, by absolute convergence, we can re-arrange and re-group the terms arbitrarily.
We will use the following first terms of the power series expansion in our computations:
\begin{equation}\label{eq:expansion}
F(p, z) = p z + \frac{p (1-p)}2 z^2+ \frac{ p(2p^2-3p+1)}6 z^3+ \frac {p(-6p^3+12p^2-7p+1)}{24}z^4+
\mathcal{O}(z^5).
\end{equation}
The following claim has already been proven in \cite{KaLS19a}, Lemma 1. 
\begin{lemma}\label{taylor}
\begin{equation*}
F(p, z)=  p \sum_{k=1}^{\infty} \frac{P_k(p)}{k!} z^k,
\end{equation*}
where  $P_k(p)$ is a power series in $p$ with constant term $P_k(0)=1$ for all $k\in\N$.
\end{lemma}

\begin{corollary}\label{coro_coefficients}
If $z$ and $y$ range in a compact subset of $\mathbb R$ and $p,\gamma$ are such that $(p,\gamma (\pm z+y)) \in \mathcal D$, then
\begin{enumerate}
\item[(i)]
$$\frac{F(p,\gamma (z+y))+F(p,\gamma (-z+y))}{2}= p\gamma y + \frac{1}{2} p(1-p) \gamma^2 (z^2+y^2) + \mathcal O(p\gamma^3),$$
\item[(ii)]
$$\frac{F(p,\gamma (z+y))-F(p,\gamma (-z+y))}{2}= p\gamma z +p (1-p)\gamma^2 zy+ \mathcal O(p\gamma^3)$$
\end{enumerate}
In both cases, the constant in the $\mathcal O$-term is uniform as long as $z,y$ stay in compact subsets of $\mathbb R$ and $(p,\gamma (\pm z +y))  \in \mathcal D$. 
\end{corollary}
\begin{proof}
From Lemma \ref{taylor} we have
\begin{align*}
&\frac{F(p,\gamma (z+y))+F(p,\gamma (-z+y))}{2}= \frac{p}{2} \sum_{k=1}^{\infty} \frac{P_{k}(p)}{k!} \gamma^{k} \left((z+y)^k + (-z+y)^k \right)
\\
&=   p\gamma y + \frac{1}{2} p(1-p) \gamma^2 (z^2+y^2) + \frac{p}{2}\gamma^3\sum_{k=3}^{\infty}\frac{P_{k}(p)}{k!} \gamma^{k-3} \left((z+y)^k + (-z+y)^k \right)
\\
&=  p\gamma y + \frac{1}{2} p(1-p) \gamma^2 (z^2+y^2) + \mathcal O(p\gamma^3).
\end{align*}
Similarly, we have
\begin{align*}
&\frac{F(p,\gamma (z+y))-F(p,\gamma (-z+y))}{2}= \frac{p}{2} \sum_{k=1}^{\infty} \frac{P_{k}(p)}{k!} \gamma^{k} \left((z+y)^k - (-z+y)^k \right)
\\&=p\gamma z +p (1-p)\gamma^2 zy+ \mathcal O(p\gamma^3).
\end{align*}
\end{proof}
\section{Proof of Theorem \ref{theo:b>1}}\label{sec_proof_th_1-1}
In this section we will prove Theorem \ref{theo:b>1}.
The proof of Theorem \ref{theo:clt_magn_field} is quite similar and in  Section \ref{sec:proof_h_positive} we will basically point out the differences between the two proofs.
Hence, throughout this section we will assume that $h=0$ and $\beta>1$.
The main idea, for both Theorem \ref{theo:b>1} and Theorem \ref{theo:clt_magn_field}, is to consider the following generalization of the partition function
\begin{equation*}
Z_N^+(\beta, g)
:=
\sum_{\sigma \in \{-1,+1\}^N}  g\left(\frac{|\sigma|-Nm}{\sqrt N}\right) \eee^{- \beta H(\sigma)} \ind_{|\sigma|>0},
\end{equation*}
where we put $m:= m^+(\beta)$.
Then we have
\begin{equation}\label{eq_int_against_L_N}
\frac 12 \int_{0}^{+\infty} g(x) L_N^+(d x)
=
\E_{\mu_\beta }\left[ g\left( \frac{|\sigma|-Nm}{\sqrt N} \right) \ind_{|\sigma|>0}\right]
=
\frac{  Z_N^+(\beta, g)}{  Z_N(\beta)},
\end{equation}
 where, for a fixed disorder  $ (\varepsilon_{i,j})_{i,j=1}^N$, $\E_{\mu_\beta }$ denotes the expectation with respect to
the Gibbs measure $\mu_{\beta}$.

Instead of $ Z_N^+(\beta, g)$ we study its normalized version $\tilde Z_N^+(\beta, g)$ in which the summands, as we will show, behave like asymptotically independent random variables.
To define $\tilde Z_N^+(\beta, g)$, for $\sigma \in \{\pm 1\}^N$ and for fixed $\beta >0$ we first introduce
\begin{equation*}
\gamma:= \frac{\beta}{2 Np}
\end{equation*}
as well as
\begin{equation*}
T(\sigma) := \exp\left(\gamma \sum_{i,j=1}^N \vep_{i,j} \sigma_i \sigma_j -\gamma m^2 \sum_{i,j=1}^N \vep_{i,j}\right).
\end{equation*}
Then, for $g \in \mathcal{C}^b(\R)$ (meaning that $g$ is globally bounded and continuous) such that  $g \ge 0, g \not \equiv 0$, we set
\begin{equation*}
\tilde Z_N^+(\beta, g):= \sum_{\sigma \in \{-1, +1\}^N} g\left(\frac{|\sigma|-Nm}{\sqrt N}\right)T(\sigma)  \ind_{|\sigma|>0}.
\end{equation*}
Note that in~\eqref{tildeZN(g)_2} we defined a related quantity
$$
\tilde Z_N(\beta):= \sum_{\sigma \in \{-1, +1\}^N} \eee^{-\beta H(\sigma)}\exp\left( -\gamma m^2 \sum_{i,j=1}^N \vep_{i,j} \right),
$$
which does not exactly correspond to $\tilde Z_N^+(\beta,1)$ since in the latter quantity the summation is restricted to $|\sigma|>0$.
Since
\begin{equation}\label{eq:relation_Z_tilde_and_Z}
Z_N^+(\beta, g)= \tilde Z_N^+(\beta, g) \exp \left(  \gamma m^2 \sum_{i,j=1}^N \varepsilon_{i,j}\right)
\end{equation}
we have
\begin{equation}\label{eq:quotien_Z_tilde}
\frac{  Z_N^+(\beta, g)}{  Z_N(\beta)}= \frac{ \tilde Z_N^+(\beta, g)}{\tilde Z_N(\beta)}.
\end{equation}
Let us investigate, how
$\tilde Z_N^+(\beta,g)$ behaves. We start our analysis
by the following computation
\begin{lemma}\label{ET}
For $h=0$, $\beta >1$, and  $p^3 N \to\infty$ as $N\to\infty$ we have for all $\sigma\in \{-1,+1\}^N$
\begin{equation*}
\E T(\sigma)=
\exp\left(\frac{\beta}{2N}(|\sigma|^2-m^2N^2)+ \frac{(1-p)\beta^2}{8 p}\left(m^4-\frac{2m^2|\sigma|^2}{N^2}+1\right)+o(1)\right)
\end{equation*}
with an $o(1)$-term that is uniform over $\sigma\in \{-1,+1\}^N$.
\end{lemma}

\begin{proof}
We compute
$$
\E T(\sigma)
=  \prod_{i,j=1}^N \E \left[ \eee^{\gamma \vep_{i,j} \sigma_i \sigma_j- \gamma m^2 \vep_{i,j}}\right]
=\prod_{i,j=1}^N\left(1-p + p \eee^{\gamma (\sigma_i \sigma_j-m^2)}\right).
$$
If we introduce
$$f(x) = f(x; p,\gamma) = \log (1-p + p\eee^{\gamma (x-m^2)})=F(p,\gamma (x-m^2)),$$
where is $F$ given by \eqref{eq:centralfunc}, we can continue by
$$
 \E T(\sigma)
=
\exp\left(\sum_{i,j=1}^N \log (1-p + p\eee^{\gamma (\sigma_i \sigma_j-m^2)})\right)
=\exp\left(\sum_{i,j=1}^N f(\sigma_i \sigma_j)\right).
$$
Note that $\sigma_i\in \{\pm 1\}$ for all $i$ and hence $\sigma_i \sigma_j \in \{\pm 1\}$.
For the two values $\pm1$ we can rewrite $f$ in a linear form.
More precisely, we write
$$
f(x) = a_0 + a_1 x, \qquad x \in \{-1,+1\}.
$$
The two coefficients $a_0$ and $a_1$ naturally are dependent on $p$, $m$ and $\gamma$. They can be computed from the following two equations
\begin{align*}
a_0
&= \frac  {f(1) + f(-1)}{2}
= \frac{\log (1-p + p\eee^{\gamma(1-m^2)}) + \log (1-p + p\eee^{\gamma(-1-m^2)})}{2}
,\\
a_1
&=
\frac {f(1)-f(-1)}2= \frac{\log (1-p + p\eee^{\gamma(1-m^2)}) - \log (1-p + p\eee^{\gamma(-1-m^2)})}{2}
\end{align*}
to obtain
$$
 \E T(\sigma)
= \exp\left(N^2 a_0 +a_1 |\sigma|^2\right).
$$
Using Corollary \ref{coro_coefficients} we see immediately  that
$$
a_0= -\gamma p m^2+ \frac{\gamma^2}{2}p(1-p)(m^4+1)+\mathcal{O}(p\gamma^3)
$$
and
$$
a_1=p \gamma - \gamma^2 p (1-p)m^2 +\mathcal{O}(p\gamma^3).
$$
Thus
$$
\E T(\sigma)=
\exp\left(\frac{\beta}{2N}(|\sigma|^2-m^2N^2)+ \frac{(1-p)\beta^2}{8 p}\left(m^4-\frac{2m^2|\sigma|^2}{N^2}+1\right)+o(1)\right)
$$
with an $o$-term that is uniform in $\sigma \in\{-1,+1\}^N$. This was the assertion.
\end{proof}

We will now compute the asymptotic expectation of $\tilde Z_N^+(\beta,g)$.
To this end, we will introduce the following set of spin configurations. Set
\begin{equation*}
S^1_N :=\left\{\sigma \in \{\pm 1\}^N : \left|\, |\sigma|-N m \,\right| \le \sqrt N \kappa_N \right\}
\end{equation*}
with $\kappa_N = p\sqrt N/ (p^3N)^{2/5}$.
These spin configurations will be called typical in the following proof. The corresponding set of values of $|\sigma|$ is denoted by $W_{N,m}$, i.e.
\begin{equation}\label{eq:WNm}
W_{N,m} :=\{|\sigma| : \sigma \in S_N^1\} =  \{\lceil N m- \sqrt N \kappa_N \rceil, \ldots, \lfloor N m +
\sqrt N \kappa_N \rfloor\}.
\end{equation}
The set of the atypical spin configurations is denoted by
$$
S_{N}^{1c}:=\{\sigma: \left|\, |\sigma| - N m\right| > \sqrt N \kappa_N, |\sigma|>0\}.
$$
\begin{proposition}\label{prop:expect}
For all $g \in \mathcal{C}^b(\R), g \ge 0, g \not\equiv 0$, $h=0, \beta>1$, and $p$ with $Np^3 \to \infty$ we have
\begin{equation*}
\lim_{N \to \infty} \frac{\E \tilde Z_N^+(\beta, g)}{\eee^{\frac{(1-p)\beta^2}{8 p}(1-m^4)-N I(m)} 2^{N+1}\frac{1}{\sqrt{1-m^2}} \sigma(\beta)  \E_{\xi} [g(\xi)]}=1.
\end{equation*}
Here, the function $I$ is given by formula \eqref{def:I} below and
$$
\E_{\xi} [g(\xi)]=\frac 1 {\sqrt{2 \pi \sigma^2(\beta)}}\int_{-\infty}^\infty g(x) \eee^{-\frac{x^2}{2\sigma^2(\beta)}} dx,
$$
i.e.\
$\xi$ denotes a normally distributed random variable with expectation $0$ and variance $\sigma^2(\beta)$.
\end{proposition}

\begin{proof}
By decomposing $\{\pm 1\}^N$ into typical and atypical $\sigma$'s, defined by $ S^1_N$,  we have
\begin{equation*}
\E \tilde Z_N^+(\beta, g)
=
\sum_{\sigma \in S^1_N}g\left(\frac{|\sigma|-Nm}{\sqrt N}\right)\E T(\sigma)
+
\sum_{\sigma \in S_N^{1c}} g\left(\frac{|\sigma|-Nm}{\sqrt N}\right)\E T(\sigma).
\end{equation*}
For the typical configurations with $|\sigma| \in W_{N,m}$ we have from Lemma \ref{ET}
\begin{eqnarray}\label{eq:first step exp}
&&\sum_{\sigma \in S^1_N}g\left(\frac{|\sigma|-Nm}{\sqrt N}\right)\E T(\sigma)\\
 &=&\eee^{\frac{(1-p)\beta^2}{8 p}(m^4+1)}\sum_{\sigma \in S^1_N}g\left(\frac{|\sigma|-Nm}{\sqrt N}\right)\eee^{\frac{\beta}{2N}(|\sigma|^2-m^2N^2)-\frac{(1-p)\beta^2}{8 p}\frac{2m^2|\sigma|^2}{N^2}+o(1)}\nonumber\\
&=&\eee^{\frac{(1-p)\beta^2}{8 p}(m^4+1)} \sum_{k\in W_{N,m}} \sum_{\sigma \in \{\pm 1\}^N: \atop |\sigma|=k}
g\left(\frac{|\sigma|-Nm}{\sqrt N}\right)\eee^{\frac{\beta}{2N}(|\sigma|^2-m^2N^2)-\frac{(1-p)\beta^2}{8 p}\frac{2m^2|\sigma|^2}{N^2}+o(1)}\nonumber\\
&=&\eee^{\frac{(1-p)\beta^2}{8 p}(m^4+1)} \sum_{k\in W_{N,m}}
g\left(\frac{k-Nm}{\sqrt N}\right)\eee^{\frac{\beta}{2N}(k^2-m^2N^2)-\frac{(1-p)\beta^2}{8 p}\frac{2m^2k^2}{N^2}+o(1)}\binom{N}{\frac{N+k}2}.\nonumber
\end{eqnarray}
Note that from Stirling's formula
$$
\log (n!)= n \log n - n +\frac 12 \log(2 \pi) + \frac 12 \log \left(n \right)+\mathcal{O}(1/n),
$$
we obtain for $k \in W_{N,m}$
\begin{align*}
2^{-N}\binom{N}{\frac{N+k}2}&= \sqrt{\frac {2} {\pi N}}\frac{1}{\sqrt{(1-\frac kN)(1+\frac kN)}} \eee^{-NI\left(\frac kN\right)+o(1)}\\
&= (1+o(1))\sqrt{\frac {2} {\pi N}}\frac{1}{\sqrt{(1-\frac{k^2}{N^2})}} \eee^{-NI\left(\frac kN\right)}
\\&
= (1+o(1))\sqrt{\frac {2} {\pi N (1-m^2)}} \eee^{-NI\left(\frac kN\right)}
\end{align*}
with an $o(1)$-term that is uniform for all $k$ such that $k \in W_{N,m}$. We have used that $\kappa_N = o(\sqrt N)$.
Here,
\begin{equation}\label{def:I}
I(x):=\frac{1-x}2\log (1-x)+\frac{1+x}2\log (1+x) \qquad \mbox{for } x\in [-1,1].
\end{equation}
Hence,
we arrive at
\begin{align*}
&\sum_{\sigma \in S^1_N}g\left(\frac{|\sigma|-Nm}{\sqrt N}\right)\E T(\sigma)\\& \quad =
(1+o(1))\eee^{\frac{(1-p)\beta^2}{8 p}(m^4+1)} 2^N \frac{1}{\sqrt{1-m^2}} \sqrt{\frac {2} {\pi N} }\\
& \qquad \times
\sum_{k\in W_{N,m}}
g\left(\frac{k-Nm}{\sqrt N}\right)\eee^{\frac{\beta}{2N}(k^2-m^2N^2)-\frac{(1-p)\beta^2}{8 p}\frac{2m^2k^2}{N^2}-NI\left(\frac kN\right)}.
\end{align*}
Let us write $\frac kN= m+\frac{c_k}{\sqrt N}$ with $|c_k|\le \kappa_N$. Then,
$$
NI\left(\frac kN\right)= NI(m)+I'(m)c_k \sqrt N+ I''(m)\frac{c_k^2}2+\mathcal{O}\left(\frac{\kappa_N^3}{\sqrt N}\right)
$$
and the $\mathcal{O}$-term is uniform for all $k\in  W_{N,m}$ and can be estimated above by $o(1)$
because of our choice of $\kappa_N$ and since $p^3N\to\infty$.
Then
\begin{align}
&\sum_{\sigma \in S^1_N}g\left(\frac{|\sigma|-Nm}{\sqrt N}\right)\E T(\sigma) \nonumber\\
& \quad =
(1+o(1))\eee^{\frac{(1-p)\beta^2}{8 p}(m^4+1)-N I(m)} 2^N \frac{1}{\sqrt{1-m^2}} \sqrt{\frac {2} {\pi N} } \nonumber\\
&\qquad
\times \sum_{k\in W_{N,m}}
g\left(\frac{k-Nm}{\sqrt N}\right)\eee^{\frac{\beta}{2N}\left(2 N^{3/2}m c_k+c_k^2 N\right)-\frac{(1-p)\beta^2}{8 p}\frac{2m^4N^2+4m^3N^{3/2}c_k+2m^2c_k^2 N}{N^2}} \nonumber \\
&\qquad \qquad \qquad \eee^{-I'(m) c_k\sqrt N
- I''(m)\frac{c_k^2}2}. \label{eq:coeff_c_k}
\end{align}
Now  the linear term in $c_k$ in the above expression is
\begin{equation*}
c_k \left(\beta m \sqrt N -\frac{\beta^2 (1-p)m^3}{2p \sqrt N}-\sqrt N I'(m)\right)
=
c_k\left(\beta m \sqrt N -\sqrt N I'(m)\right) + o(1)
=
o(1),
\end{equation*}
where the first equality follows from $\kappa_N = o(p\sqrt N)$ while the second equality follows from
$I'(x)=\frac 12 \log\left(\frac{1+x}{1-x}\right)= \mathrm{artanh}(x)$
and  $m=\tanh(\beta m)$.
Therefore, with an $o(1)$-term that is uniform for typical $\sigma$
we have that
\begin{align*}
&\sum_{\sigma \in S^1_N}g\left(\frac{|\sigma|-Nm}{\sqrt N}\right)\E T(\sigma)\notag\\
& \quad =
(1+o(1))\eee^{\frac{(1-p)\beta^2}{8 p}(m^4+1)-N I(m)} 2^N \frac{1}{\sqrt{1-m^2}} \sqrt{\frac {2} {\pi N} }\notag\\
&\qquad
\sum_{k\in W_{N,m}}
g\left(\frac{k-Nm}{\sqrt N}\right)
\eee^{\frac{\beta}{2}c_k^2-\frac{(1-p)\beta^2}{8 p}\frac{2m^4N^2+2m^2c_k^2 N}{N^2}}\eee^{- I''(m)\frac{c_k^2}2} \notag\\
& \quad =
(1+o(1))\eee^{\frac{(1-p)\beta^2}{8 p}(-m^4+1)-N I(m)} 2^N \frac{1}{\sqrt{1-m^2}} \sqrt{\frac {2} {\pi N} }\\
&\qquad
\sum_{k\in W_{N,m}}
g\left(\frac{k-Nm}{\sqrt N}\right)\eee^{\frac{\beta-I''(m)}{2}c_k^2}\notag,
\end{align*}
where we used that $\kappa_N^2 = o(pN)$.
Note that for the term in the last exponential we have
$$
\frac{\beta-I''(m)}{2}c_k^2= -\frac{1-\beta(1-m^2)}{2(1-m^2)}c_k^2
=
-\frac{1}{2 \sigma(\beta)^2} c_k^2 = -\frac{1}{2 \sigma(\beta)^2} \left(\frac{k-Nm}{\sqrt{N}}  \right)^2,
$$
where we used $c_k = (k - Nm)/\sqrt N$. We see that
$$
\sqrt{\frac {1} {2\pi N \sigma^2(\beta)} }
\sum_{k\in W_{N,m}}
g\left(\frac{k-Nm}{\sqrt N}\right)\exp\left(-\frac{1}{2 \sigma(\beta)^2} \left(\frac{k-Nm}{\sqrt{N}}  \right)^2\right)
$$
converges to
$$
\sqrt{\frac {1} {2\pi \sigma^2(\beta)} } \int_{-\infty}^{\infty}g(x)\exp\left(- \frac{x^2}{2 \sigma(\beta)^2}\right) dx= \E_{\xi} [g(\xi)].$$
Indeed, this is basically the approximation of an integral by its Riemann sum. Note that for $k\in W_{N,m}$ the variable  $c_k = (k - Nm)/\sqrt N$
ranges in $[-\kappa_N, +\kappa_N]$ intersected with a lattice of mesh size $1/\sqrt N$. Over each fixed interval, the Riemann approximation argument applies. Since $\kappa_N\to\infty$, we need an additional justification for the applicability of the Riemann approximation over intervals of growing size.
Note that
$$
\sqrt{\frac {1} {2\pi \sigma^2(\beta)} } \int_{-\infty}^{\infty}g(x)
\exp\left(- \frac{x^2}{2 \sigma(\beta)^2}\right) dx < \infty.
$$
Hence, for all $\vep>0$, there is a compact interval $I_\vep$ such that
$$
\sqrt{\frac {1} {2\pi \sigma^2(\beta)} }  \int_{x \notin I_\vep }g(x)\exp\left(- \frac{x^2}{2 \sigma(\beta)^2}\right) dx< \vep
$$
as well as
$$
\sqrt{\frac {1} {2\pi N \sigma^2(\beta)} }
\sum_{k\in W_{N,m}:\atop c_k\in I_\vep^c }
g\left(\frac{k-Nm}{\sqrt N}\right)\exp\left(-\frac{1}{2 \sigma(\beta)^2} \left(\frac{k-Nm}{\sqrt{N}}\right)^2\right)< \vep,
$$
where $I_\vep^c$ denotes the complement of $I_\vep$ in $\R$. For the second claim, bound $g$ by its supremum and estimate the remaining sum by the corresponding integral using monotonicity.
On the fixed interval $I_\vep$ we have convergence of Riemann sums to Riemann integrals, meaning that for sufficiently
large $N$, the difference between both is at most $\vep$.
But this means
\begin{multline*}
\left| \sqrt{\frac {1} {2\pi N \sigma^2(\beta)} }
\sum_{k\in W_{N,m}}
g\left(\frac{k-Nm}{\sqrt N}\right)\exp\left(-\frac{1}{2 \sigma(\beta)^2} \left(\frac{k-Nm}{\sqrt{N}}\right)^2\right)\right.
\\ - \left.
\sqrt{\frac {1} {2\pi \sigma^2(\beta)} }  \int_{-\infty}^{\infty}g(x)\exp\left(- \frac{x^2}{2 \sigma(\beta)^2}\right) dx \right| < 3\vep
\end{multline*}
for all $N$ sufficiently large.

To prove the proposition, it suffices to show that the atypical spin configurations do not contribute to the asymptotic size of $\tilde Z_N^+(\beta,g)$.
Recall that the set of atypical spin configurations is denoted by
$$
S_{N}^{1c}:=\{\sigma: \left|\, |\sigma| - N m\right| > \sqrt N \kappa_N, |\sigma|>0\}.
$$
The corresponding set of possible values for $|\sigma|$ will be called $W_{N,m}^c$.
We will use the following bound on the binomial coefficient which again is
a consequence of Markov's inequality:
\begin{equation}\label{eq:crude_bound_1}
\binom{N}{\frac{N+k}2} \le
2^N \eee^{-N I(\frac k N)},
\qquad |k|\leq N.
\end{equation}
We can use Lemma~\ref{ET} in combination with~\eqref{eq:crude_bound_1} to obtain
\begin{align}
&\sum_{\sigma \in S^{1c}_{N}}g\left(\frac{|\sigma|-Nm}{\sqrt N}\right)\E T(\sigma) \nonumber\\
\le &\eee^{\frac{(1-p)\beta^2}{8 p}(m^4+1)}||g||_{\infty}\sum_{k \in W_{N,m}^c}
\eee^{\frac{\beta}{2N}(k^2-m^2N^2)-\frac{(1-p)\beta^2}{8 p}\frac{2m^2 k^2}{N^2}+o(1)}\binom{N}{\frac{N+k}2} \nonumber \\
\le& C \eee^{\frac{(1-p)\beta^2}{8 p}(m^4+1)}2^N \sum_{
k \in W_{N,m}^c}
\eee^{\frac{\beta}{2N}(k^2-m^2N^2)-\frac{(1-p)\beta^2}{8 p}\frac{2m^2k^2}{N^2}-NI(\frac kN)} \label{eq:atyp_spin_1}
\end{align}
for some constant $C>0$.
For the next step again write $k= N m + c_k \sqrt N$, this time $\kappa_N < |c_k| < C_m\sqrt N$ (the exact value for $C_m$ depends on $m$).
For the second term in the exponent we obtain the estimate
\begin{align*}
&-\frac{(1-p)\beta^2}{8 p}\frac{2m^2k^2}{N^2}= -\frac{m^4(1-p)\beta^2}{4 p}
-\frac{c_k (1-p)\beta^2}{2 p \sqrt N}m^3-\frac{c_k^2(1-p)\beta^2}{4 p N}m^2
\\
& \leq  -\frac{m^4(1-p)\beta^2}{4 p}-\frac{c_k (1-p)\beta^2}{2 p \sqrt N}m^3.\nonumber
\end{align*}
Inserting this estimate into \eqref{eq:atyp_spin_1} leads to
\begin{align}
&\sum_{\sigma \in S^{1c}_{N}}g\left(\frac{|\sigma|-Nm}{\sqrt N}\right)\E T(\sigma) \nonumber\\
\le& C \eee^{\frac{(1-p)\beta^2}{8 p}(m^4+1)}2^N \sum_{
k \in W_{N,m}^c}
\eee^{\frac{\beta}{2N}(k^2-m^2N^2)-\frac{m^4(1-p)\beta^2}{4 p}-\frac{c_k (1-p)\beta^2}{2 p \sqrt N}m^3-NI(\frac kN)}. \label{eq:atyp_spin_2}
\end{align}
Next observe that from the analysis of the Curie-Weiss model (see \cite{Ellis-EntropyLargeDeviationsAndStatisticalMechanics} for the large deviations regime and \cite{EL04} for the regime of moderate deviations) we know that the function
$\frac{k}{N} \mapsto \frac{\beta}{2N}k^2-NI(\frac kN)$ attains its maximum for $\frac kN$ positive at $m$ and that
\begin{equation}\label{eq:Iaussen}
\frac{\beta}{2N}k^2-NI\left(\frac kN\right)\le N\left(\frac{\beta}{2}m^2-I(m)\right)-K_1 c_k^2
\end{equation}
for some sufficiently small constant $K_1>0$.
Moreover,  $c_k^2$ will be at least of order $\kappa_N^2$.
Hence
\begin{align}
&\sum_{\sigma \in S^{1c}_{N}}g\left(\frac{|\sigma|-Nm}{\sqrt N}\right)\E T(\sigma)
\notag \\
\le &
C \eee^{\frac{(1-p)\beta^2}{8 p}(-m^4+1)}2^N
\eee^{-N I(m)}\sum_{
k \in W_{N,m}^c}
\eee^{-\frac{2c_k (1-p)\beta^2}{4 p \sqrt N}m^3-K_1 c_k^2}
\notag \\
\le &
C \eee^{\frac{(1-p)\beta^2}{8 p}(-m^4+1)}2^N
\eee^{-N I(m)}\sum_{
k \in W_{N,m}^c}
\eee^{-K_2 c_k^2}\label{eq:expectaussen}
\end{align}
for some other constant $K_2>0$. We used that $c_k/(p\sqrt N) = c_k^2 /(c_k p\sqrt N)$ with  $c_kp\sqrt N\to\infty$ in the denominator.

Now $c_k^2\ge \kappa_N^2\geq N^{1/10}$ and the sum contains at most $N$ summands, which yields
\begin{equation}\label{eq:expectaussen2}
\lim_{N \to \infty}\sum_{
k \in W_{N,m}^c}
\eee^{-K_2 \kappa_N^2} = 0.
\end{equation}
This shows that the contribution of the spin configurations in $S_{N}^{1c}$ is negligible and therefore proves the proposition.
\end{proof}

In the next step we will control the variance of $\tilde Z_N^+(\beta, g)$ in order to show that it is of smaller order than the squared expectation. This would imply that the quantity $\tilde Z_N^+(\beta, g)$ is self-averaging meaning that $\tilde Z_N^+(\beta, g)/ \E \tilde Z_N^+(\beta, g)$ converges in probability to $1$. Our first step in this direction is
\begin{lemma}\label{CovT}
For $h=0$, $\beta>1$, all $p=p(N)$ such that $p^3 N \to \infty$, and all $\sigma, \tau \in \{-1,+1\}^N$ we have
\begin{equation*}
\E (T(\sigma)T(\tau))= \exp(N^2 b_0 + b_1 |\sigma|^2 + b_2 |\tau|^2 + b_{12} |\sigma \tau|^2),
\end{equation*}
where
\begin{align*}
b_0 &= -2 (m^2 p) \gamma+(p+2 m^4 p-p^2-2 m^4 p^2) \gamma^2+\mathcal{O}(p\gamma^3),
\\
b_1 =b_2&=p \gamma
+(-2 m^2 p+2 m^2 p^2) \gamma^2+\mathcal{O}(p\gamma^3),
\\
b_{12} &=  (p - p^2) \gamma^2 +\mathcal{O}(p\gamma^3),
\end{align*}
and the $\mathcal{O}$-term  is uniform over $\sigma, \tau\in \{-1,+1\}^N$.  Here, we set
$$
|\sigma\tau|:=\sum_{i=1}^N \sigma_i \tau_i.
$$
In particular, we have
\begin{multline*}
 \E (T(\sigma)T(\tau))
 = \eee^{-2m^2N^2p\gamma+N^2p(1-p)(2m^4+1)\gamma^2} \\\times \eee^{\left(p \gamma
+(-2 m^2 p+2 m^2 p^2) \gamma^2\right)(|\sigma|^2+|\tau|^2)+(p - p^2) \gamma^2 |\sigma\tau|^2+o(1)}.
\end{multline*}
\end{lemma}
\begin{proof}
We have
\begin{equation*}
\mathbb E(T(\sigma) T(\tau))=  \prod_{i,j=1}^N \left( 1-p+pe^{\gamma (\sigma_i \sigma_j + \tau_i \tau_j -2m^2)}\right)= \exp \left( \sum_{i,j=1}^N f (\sigma_i \sigma_j + \tau_i \tau_j)\right),
\end{equation*}
where
$$f(x) = f(x; p,\gamma) = \log (1-p + p\eee^{\gamma (x-2m^2)})=F(p,\gamma (x-2m^2)),$$
and $F$ is given by \eqref{eq:centralfunc}.
Since for fixed $m$, $f(\sigma_i \sigma_j + \tau_i \tau_j)$ is a function of the arguments $x_1 = \sigma_i \sigma_j$ and $x_2= \tau_i \tau_j$, where $x_1$ and $x_2$ take values in $\{\pm 1\}$.
Hence, for these values of $x_1$ and $x_2$, we can write
\begin{align*}
f(x_1+x_2)=b_0 + b_1x_1 +b_2x_2 + b_{12} x_1x_2,
\end{align*}
where the coefficients are given by
\begin{align*}
b_0 &= \frac {f(2)+f(-2)+2f(0)}{4},\\
b_{12} &= \frac {f(2)+f(-2)-2f(0)}4,\\
b_1&= b_{2} = \frac {f(2)-f(-2)}4.
\end{align*}
The representation of the coefficients is then an immediate consequence of Corollary \ref{coro_coefficients}.
Using \eqref{eq:expansion} the lemma is proved.
\end{proof}
From here we start to estimate the variance of $\tilde Z_N^+(\beta, g)$.
\begin{proposition}\label{prop:var}
For $h=0$, $\beta>1$, all $p=p(N)$ such that $p^3 N \to \infty$, and all $g \in \mathcal{C}^b(\R), g \ge 0, g\not\equiv 0$ we have that
\begin{equation*}
\V(\tilde Z_N^+(\beta, g))=o\left(\E^2[\tilde Z_N^+(\beta, g)]\right).
\end{equation*}
\end{proposition}
\begin{proof}
Obviously,
$$
\V(\tilde Z_N^+(\beta, g))= \E\left[\left(\tilde Z_N^+(\beta, g)\right)^2\right]- \E^2[\tilde Z_N^+(\beta, g)].
$$
The asymptotics of the second term on the right is already known  from Proposition~\ref{prop:expect}.
Our aim is to show that the first term satisfies
$$
\E\left[\left(\tilde Z_N^+(\beta, g)\right)^2\right]
\leq
(1+o(1))\mathbb E^2[\tilde Z_N^+(\beta,g)].
$$
Since the variance cannot become negative, this would imply the assertion.
Introduce the set of typical pairs of spin configurations
$$
S_N^2:=\{(\sigma,\tau): |\sigma|, |\tau|  \in W_{N,m}, \big| |\sigma\tau|-Nm^2 \big| \le C' \sqrt  N  \kappa_N \},
$$
where $W_{N,m}$ is defined as in \eqref{eq:WNm}, $C'$ is a large constant to be specified below,  and we recall that
$
|\sigma\tau|:=\sum_{i=1}^N \sigma_i \tau_i.
$
The pairs of spin configurations $(\sigma, \tau)$ that are not in $S_N^2$  will be called atypical. The set of atypical pairs $(\sigma,\tau)$ that  satisfy $|\sigma|>0$, $|\tau|>0$ will be denoted by $S_N^{2c}$.
We split $\E[(\tilde Z_N^+(\beta, g))^2]$ into the contribution of typical and atypical pairs of spin configurations as follows:
\begin{align*}
\E\left[\left(\tilde Z_N^+(\beta, g)\right)^2\right]&=
\sum_{(\sigma,\tau) \in \{\pm1\}^N \times \{\pm1\}^N }g\left(\frac{|\sigma|-Nm}{\sqrt N}\right)g\left(\frac{|\tau|-Nm}{\sqrt N}\right)\E (T(\sigma)T(\tau))\\
&=\sum_{(\sigma,\tau) \in S_N^2 }g\left(\frac{|\sigma|-Nm}{\sqrt N}\right)g\left(\frac{|\tau|-Nm}{\sqrt N}\right)\E (T(\sigma)T(\tau))\\
& \quad +\sum_{(\sigma,\tau) \in S_N^{2c} }g\left(\frac{|\sigma|-Nm}{\sqrt N}\right)g\left(\frac{|\tau|-Nm}{\sqrt N}\right)\E (T(\sigma)T(\tau)).
\end{align*}
Let us first consider the typical spin configurations. Using Lemma \ref{CovT} we obtain
\begin{align}
&\sum_{(\sigma,\tau) \in S_N^2 }g\left(\frac{|\sigma|-Nm}{\sqrt N}\right)g\left(\frac{|\tau|-Nm}{\sqrt N}\right)\E (T(\sigma)T(\tau))\nonumber \\
&=
\sum_{(\sigma,\tau) \in S_N^2 }g\left(\frac{|\sigma|-Nm}{\sqrt N}\right)g\left(\frac{|\tau|-Nm}{\sqrt N}\right)
\eee^{-2m^2N^2p\gamma+N^2p(1-p)(2m^4+1)\gamma^2+o(1)}\nonumber \\
&\hskip 1.5cm \times \exp\left(\left(p \gamma
+(-2 m^2 p+2 m^2 p^2) \gamma^2\right)(|\sigma|^2+|\tau|^2)+(p - p^2) \gamma^2 |\sigma\tau|^2\right)\nonumber \\
&= \eee^{\frac{(1-p)\beta^2(1+2m^4)}{4p}+o(1)}\sum_{(\sigma,\tau) \in S_N^2 }g\left(\frac{|\sigma|-Nm}{\sqrt N}\right)g\left(\frac{|\tau|-Nm}{\sqrt N}\right)
\nonumber \\
&\hskip 1.5cm \times \eee^{\frac{\beta}{2N}\left((|\sigma|^2-m^2N^2)+(|\tau|^2-m^2N^2)\right)-\frac{(1-p)\beta^2}{4 p}\frac{2m^2(|\sigma|^2+|\tau|^2)}{N^2}+\frac{(1-p)\beta^2}{4N^2p} |\sigma\tau|^2}. \label{eq:est_var1}
\end{align}
Our first observation is that the term with $|\sigma\tau|^2$ can be asymptotically replaced by a term not depending on $\sigma$ and $\tau$. Indeed, for $(\sigma,\tau) \in S_N^2$ we have that
$$
\frac{(1-p)\beta^2}{4N^2p} |\sigma\tau|^2=\frac{(1-p)\beta^2 m^4}{4p}+o(1)
$$
because $\kappa_N = o(pN)$.
Our second observation is that the same can be done for the term involving $|\sigma|^2 + |\tau|^2$ since
$$
\frac{(1-p)\beta^2}{4 p}\frac{2m^2(|\sigma|^2+|\tau|^2)}{N^2}
=
\frac{(1-p)\beta^2}{4 p}\frac{2m^2\cdot 2 N^2m^2}{N^2} +o(1)
=
 \frac{(1-p)\beta^2m^4}{p} +o(1),
$$
where we used that $\kappa_N = o(p\sqrt N)$.
Overall, we get
\begin{multline*}
\sum_{(\sigma,\tau) \in S_N^2 }g\left(\frac{|\sigma|-Nm}{\sqrt N}\right)g\left(\frac{|\tau|-Nm}{\sqrt N}\right)\E (T(\sigma)T(\tau))
\\=
\eee^{\frac{(1-p)\beta^2(1 - m^4)}{4p}+o(1)}\sum_{(\sigma,\tau) \in S_N^2 }g\left(\frac{|\sigma|-Nm}{\sqrt N}\right)g\left(\frac{|\tau|-Nm}{\sqrt N}\right)
\\ \times \eee^{\frac{\beta}{2N}\left((|\sigma|^2-m^2N^2)+(|\tau|^2-m^2N^2)\right)}.
\end{multline*}
Now we have the upper estimate
\begin{align*}
&\sum_{(\sigma,\tau) \in S_N^2 }g\left(\frac{|\sigma|-Nm}{\sqrt N}\right)g\left(\frac{|\tau|-Nm}{\sqrt N}\right)\E (T(\sigma)T(\tau))
\\
&\leq
\eee^{\frac{(1-p)\beta^2(1 - m^4)}{4p}+o(1)}
\sum_{(\sigma,\tau) \in S_N^1 \times S_N^1}g\left(\frac{|\sigma|-Nm}{\sqrt N}\right)g\left(\frac{|\tau|-Nm}{\sqrt N}\right)
\\ & \quad \times
\eee^{\frac{\beta}{2N}\left((|\sigma|^2-m^2N^2)+(|\tau|^2-m^2N^2)\right)}
\\
&\leq \eee^{\frac{(1-p)\beta^2(1 - m^4)}{4p}+o(1)}
\left(\sum_{\sigma \in S_N^1}g\left(\frac{|\sigma|-Nm}{\sqrt N}\right)
\eee^{\frac{\beta}{2N} (|\sigma|^2-m^2N^2)}\right)^2.
\end{align*}
To justify the inequality, observe that in the first line we sum over a smaller set of pairs $(\sigma, \tau)$ because $S_N^2$ involves an additional constraint on $|\sigma\tau|$, and recall that $g\geq 0$.

Now we proceed similarly to the proof of Proposition \ref{prop:expect}. Indeed, in the same way we prove that
$$
\sum_{\sigma \in S_N^1}g\left(\frac{|\sigma|-Nm}{\sqrt N}\right)
\eee^{\frac{\beta}{2N} (|\sigma|^2-m^2N^2)}
=
\eee^{-N I(m)} 2^{N+1}\frac{1+o(1)}{\sqrt{1-m^2}} \sigma(\beta)  \E_{\xi} [g(\xi)].
$$
This, together with the statement of Proposition \ref{prop:expect} shows that
\begin{align*}
\sum_{(\sigma,\tau) \in S_N^2 }g\left(\frac{|\sigma|-Nm}{\sqrt N}\right)g\left(\frac{|\tau|-Nm}{\sqrt N}\right)\E (T(\sigma)T(\tau))
 \leq(1+o(1)) \left(\mathbb E\tilde Z_N^+(\beta,g)\right)^2.
\end{align*}

We will now show that the contribution of the atypical spins to the variance of $\tilde Z_N^+(\beta,g)$ is negligible.
We need to show that
$$
\sum_{(\sigma,\tau) \in S_N^{2c} }g\left(\frac{|\sigma|-Nm}{\sqrt N}\right)g\left(\frac{|\tau|-Nm}{\sqrt N}\right)\E (T(\sigma)T(\tau))=o\left(\left(\E\tilde Z_N^+(\beta,g)\right)^2\right).
$$
Note that the pairs of spin configurations $(\sigma,\tau) \in S_N^{2c}$ either satisfy
$$\left|\, |\sigma|-Nm\right| > \sqrt N \kappa_N
\, \,  \mbox{ or, } \, \,
\left|\, |\tau|-Nm\right| > \sqrt N \kappa_N
\, \,  \mbox{ or, }\, \,
\left|\, |\sigma\tau|-Nm^2\right| >  C' \sqrt  N  \kappa_N.
$$
In the case when
$$
\left|\, |\sigma|-Nm\right| > \sqrt N \kappa_N \, \,  \text{ or} \, \,  \left|\, |\tau|-Nm\right| > \sqrt N \kappa_N
$$
 we can proceed similarly as in the proof of Proposition \ref{prop:expect}. For concreteness,
 let us assume that we consider the situation where
$|\, |\sigma|-Nm| > \sqrt N \kappa_N$ and $\tau$ is
arbitrary
and let us denote the corresponding set of spin configurations by
$S_{N,A}^{2c}$. Then, starting from \eqref{eq:est_var1}, we
estimate
\begin{align*}
& \sum_{(\sigma,\tau) \in S_{N,A}^{2c} }g\left(\frac{|\sigma|-Nm}{\sqrt N}\right)g\left(\frac{|\tau|-Nm}{\sqrt N}\right)\E (T(\sigma)T(\tau)) \\
&\quad \le
\eee^{\frac{(1-p)\beta^2(1+2m^4)}{4p}+o(1)}\|g\|_{\infty} \sum_{(\sigma,\tau)  \in S_{N,A}^{2c}}
g\left(\frac{|\tau|-Nm}{\sqrt N}\right) \eee^{\frac{\beta}{2N}\left((|\sigma|^2-m^2N^2)+(|\tau|^2-m^2N^2)\right)}\\
&\hskip6.6cm \times \eee^{-\frac{(1-p)\beta^2}{4 p}\frac{2m^2(|\sigma|^2+|\tau|^2)}{N^2}+\frac{(1-p)\beta^2}{4N^2p} |\sigma\tau|^2} \\
&\leq \eee^{\frac{(1-p)\beta^2(2+2m^4)}{4p}+o(1)}\|g\|_{\infty}
\sum_{(\sigma,\tau)  \in S_{N,A}^{2c}} g\left(\frac{|\tau|-Nm}{\sqrt N}\right)
\eee^{\frac{\beta}{2N}\left((|\sigma|^2-m^2N^2)+(|\tau|^2-m^2N^2)\right)}
 \\
&\hskip6.6cm \times \eee^{-\frac{(1-p)\beta^2}{4 p}\frac{2m^2(|\sigma|^2+|\tau|^2)}{N^2}}
\end{align*}
because $\frac{(1-p)\beta^2}{4N^2p} |\sigma\tau|^2 \le  \frac{(1-p)\beta^2}{4p}$. Thus
\begin{align*}
& \sum_{(\sigma,\tau) \in S_{N,A}^{2c} }g\left(\frac{|\sigma|-Nm}{\sqrt N}\right)g\left(\frac{|\tau|-Nm}{\sqrt N}\right)\E (T(\sigma)T(\tau))
\\
&\le \eee^{\frac{(1-p)\beta^2(2+2m^4)}{4p}+o(1)}\|g\|_{\infty}
\\
&\times
\sum_{(k,l):\atop |k-Nm| > \sqrt N \kappa_N}
g\left(\frac{l-Nm}{\sqrt N}\right)
\eee^{\frac{\beta}{2N}\left((k^2-m^2N^2)+(l^2-m^2N^2)\right)-\frac{(1-p)\beta^2}{4p}
\frac{2m^2(k^2+l^2)}{N^2}}\binom{N}{\frac{N+k}2}
\binom{N}{\frac{N+l}2}
\\
&\le \eee^{\frac{(1-p)\beta^2(2+2m^4)}{4p}}2^{N} \|g\|_{\infty}
\sum_{k: |k-Nm| > \sqrt N \kappa_N}
\eee^{\frac{\beta}{2N}(k^2-m^2N^2)-\frac{(1-p)\beta^2}{4 p}\frac{2m^2k^2}{N^2}-NI\left(\frac kN\right)}\\
&\hskip5cm \times \sum_l g\left(\frac{l-Nm}{\sqrt N}\right)
\eee^{\frac{\beta}{2N}(l^2-m^2N^2)-\frac{(1-p)\beta^2}{4 p}\frac{2m^2l^2}{N^2}}\binom{N}{\frac{N+l}2}
\end{align*}
where the last step follows from
\begin{equation*}
2^{-N}\binom{N}{\frac{N+k}2}\le \exp\left(-NI\left(\frac kN\right)\right).
\end{equation*}
Note that as in Proposition \ref{prop:expect}, especially equation \eqref{eq:first step exp} and the following equations we obtain that
\begin{align*}
&\eee^{\frac{(1-p)\beta^2(2+2m^4)}{4p}}\sum_l g\left(\frac{l-Nm}{\sqrt N}\right)
\eee^{\frac{\beta}{2N}(l^2-m^2N^2)-\frac{(1-p)\beta^2}{4 p}\frac{2m^2l^2}{N^2}}\binom{N}{\frac{N+l}2}
\\ &\quad =
(1+o(1))\eee^{\frac{(1-p)\beta^2}{2p}-NI(m)}2^{N+1}\frac 1 {\sqrt{1-m^2}}\sigma(\beta)\E_\xi[g(\xi)]\\
&\quad =(1+o(1))\eee^{\frac{(1-p)\beta^2(3+m^4)}{8p}}
\E \tilde Z_N^+(\beta,g).
\end{align*}
This implies that
\begin{align*}
& \sum_{(\sigma,\tau) \in S_{N,A}^{2c} }g\left(\frac{|\sigma|-Nm}{\sqrt N}\right)g\left(\frac{|\tau|-Nm}{\sqrt N}\right)\E (T(\sigma)T(\tau)) \\
&  \le (1+o(1))\E \tilde Z_N^+(\beta,g) \eee^{\frac{(1-p)\beta^2(3+m^4)}{8p}}
2^{N} \|g\|_{\infty}
\sum_{k: |k-Nm| > \sqrt N \kappa_N}
\eee^{\frac{\beta}{2N}(k^2-m^2N^2)-\frac{(1-p)\beta^2}{4 p}\frac{2m^2k^2}{N^2}-NI\left(\frac kN\right)}\\
&  \le (1+o(1))\E \tilde Z_N^+(\beta,g) \eee^{\frac{(1-p)\beta^2(3+m^4)}{8p}}
2^{N} \|g\|_{\infty}
\sum_{k: |k-Nm| > \sqrt N \kappa_N}
\eee^{\frac{\beta}{2N}(k^2-m^2N^2)-NI\left(\frac kN\right)}.
\end{align*}
But following the steps in \eqref{eq:Iaussen}, \eqref{eq:expectaussen}, and \eqref{eq:expectaussen2} we see that
\begin{align*}
&\eee^{\frac{(1-p)\beta^2(3+m^4)}{8p}} 2^{N} \sum_{k:|k-Nm| > \sqrt N \kappa_N}
\eee^{\frac{\beta}{2N}\left(k^2-m^2N^2\right)-NI\left(\frac kN\right)}\\
=&\eee^{\frac{(1-p)\beta^2(2+2m^4)}{8p}}\eee^{\frac{(1-p)\beta^2(1-m^4)}{8p}}2^{N} \sum_{k:|k-Nm| > \sqrt N \kappa_N}
\eee^{\frac{\beta}{2N}\left(k^2-m^2N^2\right)-NI\left(\frac kN\right)}\\
&\le \eee^{\frac{(1-p)\beta^2(2+2m^4)}{8p}} \eee^{\frac{(1-p)\beta^2}{8 p}(-m^4+1)}2^N
\eee^{-N I(m)}\sum_{
k \in W_{N,m}^c}
\eee^{-K_2 c_k^2}
\end{align*}
with the set $W_{N,m}^c$ defined as in the proof of Proposition \ref{prop:expect}.
But
$$
\eee^{\frac{(1-p)\beta^2(2+2m^4)}{8p}} \sum_{k \in W_{N,m}^c} \eee^{-K_2 c_k^2}\le C \eee^{\frac{(1-p)\beta^2(2+2m^4)}{8p}}
\eee^{-K_2 \kappa_N^2} \to 0
$$
because $\kappa_N^2 p \to\infty$. Together with Proposition \ref{prop:expect} this shows that
$$
\sum_{(\sigma,\tau) \in S_{N,A}^{2c} }g\left(\frac{|\sigma|-Nm}{\sqrt N}\right)g\left(\frac{|\tau|-Nm}{\sqrt N}\right)\E (T(\sigma)T(\tau))=o\left(\left(\E \tilde Z_N^+(\beta,g)^2\right)\right).
$$
Hence the contribution of the pairs of spin configurations from $S_{N,A}^{2c}$ is asymptotically negligible.

The contributions of pairs of spin configurations from $S_{N,B}^{2c}$, the set where
$$|\, |\tau|-Nm| > \sqrt N \kappa_N $$
and  $\sigma$ is arbitrary, is bounded in the same way.

It remains to estimate the contribution of the pairs of spin configurations of the set
\begin{multline*}
S_{N,C}^{2c}:=\left\{(\sigma, \tau) \in S_{N}^{2c}:
\left|\, |\sigma|-Nm\right| \le \sqrt N \kappa_N, \right. \\ \left.
\left|\, |\tau|-Nm\right| \le \sqrt N \kappa_N,
\left|\,|\sigma\tau|-Nm^2\right| >  C' \sqrt  N  \kappa_N\right\}.
\end{multline*}
Let us denote by $R_{N,C}^{2c}$ the set of possible values $(k,l,n)$ the vector $(|\sigma|,|\tau|,|\sigma\tau|)$ can take,
when $(\sigma,\tau) \in S_{N,C}^{2c}$, formally
$$
R_{N,C}^{2c}:=\{(k,l,n):\exists (\sigma,\tau) \in S_{N,C}^{2c} \mbox{ with }(k,l,n)=(|\sigma|,|\tau|,|\sigma\tau|)\}.
$$
Moreover, denote by $V_N(k,l,n)$ the set of pairs
$$(\sigma,\tau)\in \{\pm 1\}^N\times\{\pm 1\}^N \quad \mbox{for which }|\sigma|=k,|\tau|= l, \mbox{and }|\sigma \tau| = n$$
and set $\nu_N(k,l,n) := \# V_N(k,l,n)$. Note in particular that by the definition of $|\sigma|, |\tau|$ and $|\sigma \tau|$ we have
\begin{equation}\label{eq:relationk_l_n}
-(N+k) \leq l+n \leq N+k \quad \text{and} \quad -(N-k) \leq l-n \leq N-k.
\end{equation}

In order to treat the corresponding contribution we need to compute the distribution of $|\sigma\tau|$ in greater detail.
We begin by using Lemma \ref{CovT} again:
\begin{align*}
&\sum_{(\sigma,\tau) \in S_{N,C}^{2c} }g\left(\frac{|\sigma|-Nm}{\sqrt N}\right)g\left(\frac{|\tau|-Nm}{\sqrt N}\right)\E (T(\sigma)T(\tau))\\
&=
\eee^{\frac{(1-p)\beta^2(1+2m^4)}{4p}+o(1)}\sum_{(\sigma,\tau)  \in S_{N,C}^{2c}}
g\left(\frac{|\sigma|-Nm}{\sqrt N}\right)g\left(\frac{|\tau|-Nm}{\sqrt N}\right)\\
&\qquad \quad  \times \eee^{\frac{\beta}{2N}\left((|\sigma|^2-m^2N^2)+(|\tau|^2-m^2N^2)\right)-\frac{(1-p)\beta^2}{4 p}\frac{2m^2(|\sigma|^2+|\tau|^2)}{N^2}+\frac{(1-p)\beta^2}{4N^2p} |\sigma\tau|^2}\\
&= \eee^{\frac{(1-p)\beta^2(1+2m^4)}{4p}+o(1)}
\sum_{(k,l,n)\in R_{N,C}^{2c}}\sum_{(\sigma, \tau)\in V_N(k,l,n)}g\left(\frac{|\sigma|-Nm}{\sqrt N}\right)g\left(\frac{|\tau|-Nm}{\sqrt N}\right) \\ &\qquad \times \eee^{\frac{\beta}{2N}\left((|\sigma|^2-m^2N^2)+(|\tau|^2-m^2N^2)\right)-\frac{(1-p)\beta^2}{4 p}\frac{2m^2(|\sigma|^2+|\tau|^2)}{N^2}+\frac{(1-p)\beta^2}{4N^2p} |\sigma\tau|^2}\\
&= \eee^{\frac{(1-p)\beta^2(1+2m^4)}{4p}+o(1)}
\sum_{(k,l,n)\in R_{N,C}^{2c}} g\left(\frac{k-Nm}{\sqrt N}\right)g\left(\frac{l-Nm}{\sqrt N}\right) \\ & \qquad \times
\eee^{\frac{\beta}{2N}\left((k^2-m^2N^2)+(l^2-m^2N^2)\right)-\frac{(1-p)\beta^2}{4 p}\frac{2m^2(k^2+l^2)}{N^2}+\frac{(1-p)\beta^2}{4N^2p} n^2} \nu_N(k,l,n).
\end{align*}
As observed in \cite{KaLS19b}
$\nu_N(k,l,n)$ divided by $2^{2N}$ is  a probability mass function, which can be written in terms of a  conditional probability:
\begin{align*}
2^{-2N} \nu_N(k,l,n)&= \P_{\text{unif}}(|\sigma|=k, |\tau|=l, |\sigma\tau|=n)\\
&=
 \P_{\text{unif}}(|\sigma \tau|=n \,\Big|\, |\sigma|=k, |\tau|=l)\P(|\sigma|=k)\P( |\tau|=l)\\
&= 2^{-N}\binom{N}{\frac{N+k}2}2^{-N}\binom{N}{\frac{N+l}2}\P(|\sigma \tau|=n \,\Big|\, |\sigma|=k, |\tau|=l).
\end{align*}
Here, $\P_{\text{unif}}$ denotes  the probability distribution under which $(\sigma,\tau)$ is uniformly distributed on $\{\pm 1\}^{N}\times \{\pm 1\}^{N}$. Using the hypergeometric distribution, we can express the conditional probability $\P(|\sigma \tau|=n \,\Big|\, |\sigma|=k, |\tau|=l)$ as the following fraction:
\begin{equation}\label{eq:condprob}
\P(|\sigma \tau|=n \,\Big|\, |\sigma|=k, |\tau|=l)=
\frac{\binom{\frac{N+k}2}{\frac{N+k+l+n}4}\binom{\frac{N-k}2}{\frac{N+l-k-n}4}}{\binom{N}{\frac{N+l}2}}.
\end{equation}
We will use
\begin{equation*}
c N^{-1/2} 2^N \eee^{-N I(\frac k N) - \lambda_N(k)} \leq \binom{N}{\frac{N+k}2} \le C N^{-1/2} 2^N \eee^{-N I(\frac k N) - \lambda_N(k)},
\qquad |k|\leq N
\end{equation*}
for some constants $c,C>0$ which
was shown in \cite{KaLS19b}, Eqn.~(4.11), as a consequence of Stirling's formula.
Here
$$
\lambda_N(k) :=  \frac 12 \log \left(\frac{(N+1)^2 - k^2}{N^2}\right).
$$
With this formula
we can treat the binomial coefficients in \eqref{eq:condprob} (where we bound the $\log$-correction in the exponent of the denominator by $0$) to obtain
\begin{multline*}
\P
(|\sigma \tau |=n \,\Big|\, |\sigma |=k, |\tau |=l)
\le C \sqrt{\frac{N}{(N+k)(N-k)}} \\
\times \eee^{-N\left(\frac{N+k}{2N}I(\frac{l+n}{N+k})+\frac{N-k}{2N}I(\frac{l-n}{N-k})-I(\frac{l}{N})\right)-
\lambda_{\frac{N+k}{2}}(\frac{n+l}{2}) - \lambda_{\frac{N-k}{2}}(\frac{l-n}{2})}
\end{multline*}
for some positive constant $C$.
Note that $k=N$ or $l=N$ is excluded by the definition of $R_{N,C}^{2c}$. Moreover, on $R_{N,C}^{2c}$ we have that $\sqrt{\frac{N}{(N+k)(N-k)}}\le C/\sqrt N$ (again for another $C>0$).
Thus
\begin{align*}
&\sum_{(\sigma,\tau) \in S_{N,C}^{2c} }g\left(\frac{|\sigma|-Nm}{\sqrt N}\right)g\left(\frac{|\tau|-Nm}{\sqrt N}\right)\E (T(\sigma)T(\tau))\\
&\le C \eee^{\frac{(1-p)\beta^2(1+2m^4)}{4p}}
\sum_{(k,l,n)\in R_{N,C}^{2c}}g\left(\frac{k-Nm}{\sqrt N}\right)g\left(\frac{l-Nm}{\sqrt N}\right)
\\& \qquad \times
\eee^{\frac{\beta}{2N}\left((k^2-m^2N^2)+(l^2-m^2N^2)\right)-\frac{(1-p)\beta^2}{4 p}\frac{2m^2(k^2+l^2)}{N^2}+\frac{(1-p)\beta^2}{4N^2p} n^2}
\binom{N}{\frac{N+k}2}\binom{N}{\frac{N+l}2}
\\& \quad \times
 \frac{1}{\sqrt N} \eee^{-N\left(\frac{N+k}{2N}I(\frac{l+n}{N+k})+\frac{N-k}{2N}I(\frac{l-n}{N-k})-I(\frac{l}{N})\right)-
\lambda_{\frac{N+k}{2}}(\frac{n+l}{2}) - \lambda_{\frac{N-k}{2}}(\frac{l-n}{2})}\\
&= C \eee^{\frac{(1-p)\beta^2(1+3m^4)}{4p}}
\sum_{(k,l,n)\in R_{N,C}^{2c}}g\left(\frac{k-Nm}{\sqrt N}\right)g\left(\frac{l-Nm}{\sqrt N}\right)
\\& \qquad \times
\eee^{\frac{\beta}{2N}\left((k^2-m^2N^2)+(l^2-m^2N^2)\right)-\frac{(1-p)\beta^2}{4 p}\frac{2m^2(k^2+l^2)}{N^2}+\frac{(1-p)\beta^2}{4N^2p} (n^2-m^4N^2)}
\binom{N}{\frac{N+k}2}\binom{N}{\frac{N+l}2}
\\& \quad \times
 \frac{1}{\sqrt N} \eee^{-N\left(\frac{N+k}{2N}I(\frac{l+n}{N+k})+\frac{N-k}{2N}I(\frac{l-n}{N-k})-I(\frac{l}{N})\right)-
\lambda_{\frac{N+k}{2}}(\frac{n+l}{2}) - \lambda_{\frac{N-k}{2}}(\frac{l-n}{2})}.
\end{align*}
Now, borrowing an idea from \cite{KaLS19b} we observe that
$$\frac{l}{N}= \frac{N-k}{2N} \frac{l-n}{N-k} +  \frac{N+k}{2N} \frac{l+n}{N+k},$$
i.e. $\frac{l}{N}$ is a convex combination  of $\frac{l+n}{N+k}$ and $\frac{l-n}{N-k}$ with weights $\frac{N+k}{2N}$ and $\frac{N-k}{2N}$, respectively. On the other hand, $I$ is a convex function, and, even more, considering its Taylor expansion
\begin{equation*}
NI(l/N)= \frac{l^2}{2N}+ \sum_{j\ge 2 }d_{2j}\frac{l^{2j}}{N^{2j-1}}
\end{equation*}
with positive coefficients $d_{2j}$ we see that it is a positive linear combination of convex functions. Using that $d_2=\frac 12$ we obtain
\begin{align*}
&-N\left(\frac{N+k}{2N}I\left(\frac{l+n}{N+k}\right)+\frac{N-k}{2N}I\left(\frac{l-n}{N-k}\right)-I\left(\frac{l}{N}\right)\right)
\\
=&-N \left(\frac{N+k}{2N}\sum_{j=1}^\infty d_{2j} \left(\frac{l+n}{N+k}\right)^{2j}+
\frac{N-k}{2N}\sum_{j=1}^\infty d_{2j}\left(\frac{l-n}{N-k}\right)^{2j}-\sum_{j=1}^\infty d_{2j} \left(\frac{l}{N}\right)^{2j}\right)
\\
\le& -\frac N2 \left(\frac{N+k}{2N}\left(\frac{l+n}{N+k}\right)^2+
\frac{N-k}{2N}\left(\frac{l-n}{N-k}\right)^2-\left(\frac{l}{N}\right)^2\right)
\\= &-\frac{1}{2} \frac{(Nn-lk)^2}{N (N^2-k^2)}=-\frac{1}{2} \frac{(n-\frac{lk}{N})^2}{ (N-\frac{k^2}{N})} ,
\end{align*}
where for the inequality we used that for each $j\geq 2$
$$
\frac{N+k}{2N}d_{2j} \left(\frac{l+n}{N+k}\right)^{2j}+
\frac{N-k}{2N}d_{2j}\left(\frac{l-n}{N-k}\right)^{2j}-d_{2j} \left(\frac{l}{N}\right)^{2j}\ge 0.
$$
Thus
\begin{align*}
&\eee^{-N\left(\frac{N+k}{2N}I(\frac{l+n}{N+k})+\frac{N-k}{2N}I(\frac{l-n}{N-k})-I(\frac{l}{N})\right)-
\lambda_{\frac{N+k}{2}}(\frac{n+l}{2}) - \lambda_{\frac{N-k}{2}}(\frac{l-n}{2})}
\\
\le &
\exp\left(-\frac{(n-\frac{kl}N)^2}{2(N-\frac{k^2}N)}
-
\lambda_{\frac{N+k}{2}}\left(\frac{n+l}{2}\right) - \lambda_{\frac{N-k}{2}}\left(\frac{l-n}{2}\right)
\right)\\
\le &C \exp\left(-\frac{(n-\frac{kl}N)^2}{2(N-\frac{k^2}N)}\right).
\end{align*}
We used that $\eee^{ - \lambda_{\frac{N+k}{2}}(\frac{n+l}{2}) - \lambda_{\frac{N-k}{2}}(\frac{l-n}{2})}\le C$
as by the definition of $\lambda$ and the lower bounds in \eqref{eq:relationk_l_n}
$$
\eee^{ - \lambda_{\frac{N+k}{2}}(\frac{n+l}{2})}= \left(  1+\frac{1}{2} \frac{n+l}{N+k}\right)^{-\frac{1}{2}} \leq C
$$
and
$$
\eee^{ - \lambda_{\frac{N-k}{2}}(\frac{l-n}{2})}= \left(  1+\frac{1}{2} \frac{l-n}{N-k}\right)^{-\frac{1}{2}} \leq C.
$$
Therefore
\begin{align}  \label{eq:sumovern}
&\sum_{(\sigma,\tau) \in S_{N,C}^{2c} }g\left(\frac{|\sigma|-Nm}{\sqrt N}\right)g\left(\frac{|\tau|-Nm}{\sqrt N}\right)\E (T(\sigma)T(\tau))\notag\\
&\le C \eee^{\frac{(1-p)\beta^2(1+3m^4)}{4p}}
\sum_{(k,l,n)\in R_{N,C}^{2c}}g\left(\frac{|\sigma|-Nm}{\sqrt N}\right)g\left(\frac{|\tau|-Nm}{\sqrt N}\right)
\\& \qquad \times
\eee^{\frac{\beta}{2N}\left((k^2-m^2N^2)+(l^2-m^2N^2)\right)-\frac{(1-p)\beta^2}{4 p}\frac{2m^2(k^2+l^2)}{N^2}+\frac{(1-p)\beta^2}{4N^2p} (n^2-m^4N^2)}
\binom{N}{\frac{N+k}2}\binom{N}{\frac{N+l}2}
\notag\\& \quad \times
 \frac{1}{\sqrt N} \exp\left(-\frac{(n-\frac{kl}N)^2}{2(N-\frac{k^2}N)}\right).\notag
\end{align}
Note that on $R_{N,C}^{2c}$ we have that  $\frac{kl}N$  differs from $Nm^2$ by at most $C\sqrt N\kappa_N$ for some constant $C$. On the other hand, $n$ differs from $Nm^2$ by at least $C'\sqrt N\kappa_N$ by our definition of $S_{N,C}^{2c}$. Choosing $C':=2C$
we have that $|n-\frac{kl}N|\geq C\sqrt N\kappa_N$ and hence
$$
\frac{1}{\sqrt N} \exp\left(-\frac{(n-\frac{kl}N)^2}{2(N-\frac{k^2}N)}\right)
\leq
\exp\left(-C^2 N \kappa_N^2/(2N)\right)
=
\exp\left(-C^2 \kappa_N^2/2\right).
$$
On the other hand,
$$
\frac{(1-p)\beta^2}{4N^2p} (n^2-m^4N^2) \le K/p
$$
for some constant $p$. Hence the sum over $n$ on the right hand side of \eqref{eq:sumovern} (which contains at most $N$ summands) can be bounded by
$$
N \exp\left(-C^2 \kappa_N^2/2+\frac{K}{p}\right)=o(1)
$$
The latter is true, since $\frac 1 p =o (\kappa_N^2)$, because  $p \kappa_N^2 \to \infty$ as $N \to \infty$
and $$N \exp\left(-C^2 \kappa_N^2/2\right) = o(1).$$ Thus we have that
\begin{align*}
&\sum_{(\sigma,\tau) \in S_{N,C}^{2c} }g\left(\frac{|\sigma|-Nm}{\sqrt N}\right)g\left(\frac{|\tau|-Nm}{\sqrt N}\right)\E (T(\sigma)T(\tau))\\
&\le o(1)
\sum_{k,l\in  W_{N,m}} \eee^{\frac{(1-p)\beta^2(1+3m^4)}{4p}} g\left(\frac{|\sigma|-Nm}{\sqrt N}\right)g\left(\frac{|\tau|-Nm}{\sqrt N}\right)
\\& \qquad \times
\eee^{\frac{\beta}{2N}\left((k^2-m^2N^2)+(l^2-m^2N^2)\right)-\frac{(1-p)\beta^2}{4 p}\frac{2m^2(k^2+l^2)}{N^2}+\frac{(1-p)\beta^2}{4N^2p} (n^2-m^4N^2)}
\binom{N}{\frac{N+k}2}\binom{N}{\frac{N+l}2}.
\end{align*}
Following the lines of the proof of Proposition~\ref{prop:expect}, we see that for $k,l \in W_{N,m}$ we have
$$
\frac{(1-p)\beta^2}{4 p}\frac{2m^2(k^2+l^2)}{N^2}= \frac{(1-p)\beta^2 m^4}{p}+o(1).
$$
Hence
\begin{align*}
&\sum_{k,l\in  W_{N,m}} \eee^{\frac{(1-p)\beta^2(1+3m^4)}{4p}} g\left(\frac{|\sigma|-Nm}{\sqrt N}\right)g\left(\frac{|\tau|-Nm}{\sqrt N}\right)
\\& \qquad \times
\eee^{\frac{\beta}{2N}\left((k^2-m^2N^2)+(l^2-m^2N^2)\right)-\frac{(1-p)\beta^2}{4 p}\frac{2m^2(k^2+l^2)}{N^2}}
\binom{N}{\frac{N+k}2}\binom{N}{\frac{N+l}2}\\
&= \sum_{k,l\in  W_{N,m}} \eee^{\frac{(1-p)\beta^2(1-m^4)}{4p}} g\left(\frac{|\sigma|-Nm}{\sqrt N}\right)g\left(\frac{|\tau|-Nm}{\sqrt N}\right)
\\& \qquad \times
\eee^{\frac{\beta}{2N}\left((k^2-m^2N^2)+(l^2-m^2N^2)\right)}
\binom{N}{\frac{N+k}2}\binom{N}{\frac{N+l}2}
\end{align*}
and -- as in Proposition \ref{prop:expect} -- the
sum on the right-hand side
is bounded above by
$[\mathbb E\tilde Z_N^+(\beta, g)]^2$.
Thus
$$
\sum_{(\sigma,\tau) \in S_{N,C}^{2c} }g\left(\frac{|\sigma|}{\sqrt N}\right)g\left(\frac{|\tau|}{\sqrt N}\right)\E (T(\sigma)T(\tau)) =o\left(\left(\mathbb E \tilde Z_N^+(\beta, g)\right)^2 \right),
$$
which shows that also the contribution from $S_{N,C}^{2c}$ is negligible.
This finishes the proof.
\end{proof}

We are now ready to prove Theorem \ref{theo:b>1}.
\begin{proof}[Proof of Theorem \ref{theo:b>1}]
Proposition \ref{prop:var} shows that, when $p^3 N\to \infty$, we have that $\V(\tilde Z_N^+(\beta,g))= o((\E \tilde Z_N^+(\beta, g))^2)$ for all non-negative $g \in \mathcal{C}_b(\R)$, $g\not  \equiv 0$.
This immediately implies
\begin{equation*}
\frac{ \tilde Z_N^+(\beta, g)}{\E \tilde Z_N^+(\beta, g)} \to 1
\end{equation*}
in $L^2$, for all non-negative $g \in \mathcal{C}_b(\R)$, $g\not  \equiv 0$.
By Chebyshev's inequality, this implies that
$$
\frac{ \tilde Z_N^+(\beta, g)}{\E \tilde Z_N^+(\beta, g)}\to 1
$$
in probability. Recall from~ Proposition \ref{prop:expect} that
$$
\lim_{N \to \infty} \frac{\E \tilde Z_N^+(\beta, g)}{\eee^{\frac{(1-p)\beta^2}{8 p}(1-m^4)-N I(m)} 2^{N+1}\frac{1}{\sqrt{1-m^2}} \sigma(\beta)  \E_{\xi} [g(\xi)]}=1.
$$
Arguing in the same way but without the restriction to $|\sigma|>0$ we get
\begin{equation}\label{eq:in_probab_noplus}
\frac{ \tilde Z_N(\beta)}{\E \tilde Z_N(\beta)}\to 1
\end{equation}
in probability and
\begin{equation}\label{eq:E_tilde_Z_N}
\lim_{N \to \infty} \frac{\E \tilde Z_N(\beta)}{\eee^{\frac{(1-p)\beta^2}{8 p}(1-m^4)-N I(m)} 2^{N+1}\frac{1}{\sqrt{1-m^2}} \sigma(\beta)}=2.
\end{equation}
Note that the limit equals $2$ because the summation can be split into two sums
over configurations with $|\sigma|>0$ and $|\sigma|\leq 0$ having the same asymptotic behavior.

Recall that for the convergence of the random probability measure $L_N^+$ defined in~\eqref{eq:def_L_N}
we need to consider its integral against all non-negative $g \in \mathcal{C}_b(\R)$.
Moreover recall that, according to \eqref{eq_int_against_L_N} and \eqref{eq:quotien_Z_tilde}  we have
$$
\int_{0}^{+\infty} g(x) L_N^+(d x)
=
2\E_{\mu_\beta }\left[ g\left( \frac{\sum_{i=1}^N \sigma_i-Nm}{\sqrt N} \right)\right]
=
\frac{\tilde Z_N^+(\beta, g)}{\frac 12 \tilde Z_N(\beta)}.
$$
The above claims and Slutsky's lemma yield
\begin{equation}\label{eq:proof_th_1_1}
\lim_{N\to\infty}\int_{0}^{+\infty} g(x) L_N^+(d x)
=
\lim_{N\to\infty} \frac{ \E \tilde Z_N^+(\beta, g)}{\frac 12 \E \tilde Z_N(\beta)}
=
 \E_\xi [g(\xi)]
\end{equation}
in probability, where
$\xi$ denotes a normally distributed random variable with expectation $0$ and variance $\sigma^2(\beta)$.
Hence, we have shown that the measure  $L_N^+$, considered as a random element of the
space of finite measures, converges in probability to a normal distribution
with mean $0$ and variance $\sigma^2(\beta)$. This is the assertion of Theorem \ref{theo:b>1}.
\end{proof}
\section{Proof of Theorem \ref{theo:clt_magn_field}}\label{sec:proof_h_positive}
We go through the proof of Theorem \ref{theo:b>1} and indicate where we need to adjust the arguments.
At the beginning of Section \ref{sec_proof_th_1-1} we already defined $T$ for $h \neq 0$. Now, in the definition of $T$ and in the rest of the proof we use
\begin{equation*}
m=m^+(\beta, h)
\end{equation*}
and
\begin{equation*}
T(\sigma) := \exp\left(\gamma \sum_{i,j=1}^N \vep_{i,j} \sigma_i \sigma_j
-\gamma m^2 \sum_{i,j=1}^N \vep_{i,j} + \beta h \sum_{i=1}^N \sigma_i \right)
\end{equation*}
and
\begin{equation*}
\tilde Z_N(\beta, g):= \sum_{\sigma \in \{-1, +1\}^N} g\left(\frac{|\sigma|-Nm}{\sqrt N}\right)T(\sigma).
\end{equation*}

In the expansion of $\mathbb E T(\sigma)$, we get an additional deterministic term $\exp(\beta h |\sigma|)$, i.e.~we get immediately
\begin{lemma}
Assume $h>0, \beta>0$. Then for all $\sigma\in \{-1,+1\}^N$ we have
\begin{multline*}
\E T(\sigma)=
\exp\left(\frac{\beta}{2N}(|\sigma|^2-m^2N^2)+ \frac{(1-p)\beta^2}{8 p^2}
\left(m^4-\frac{2m^2|\sigma|^2}{N^2}+1\right)\right)
\\ \times \exp\left(h\beta |\sigma| + o(1)\right)
\end{multline*}
with an $o(1)$-term that is uniform over $\sigma\in \{-1,+1\}^N$.
\end{lemma}
The analogue of Proposition \ref{prop:expect} now reads as follows.
\begin{proposition}
For all $g \in \mathcal{C}^b(\R), g \ge 0, g \not\equiv 0$, $h>0, \beta>0$, and $p$ with $Np^3 \to \infty$ we have
\begin{equation}\label{eq:E_tilde_asympt_magnetic}
\lim_{N \to \infty} \frac{\E \tilde Z_N(\beta, g)}{\eee^{\frac{(1-p)\beta^2}{8 p}(-m^4+1)-N I(m)+\beta h Nm} 2^{N+1}\frac{1}{\sqrt{1-m^2}} \sigma(\beta, h)  \E_{\xi} [g(\xi)]}=1,
\end{equation}
where $\xi$ denotes a normally distributed random variable with expectation $0$ and variance $\sigma^2(\beta,h).$
\end{proposition}
\begin{proof}
We note that, in contrast to the case $h=0$, there is an additional term $e^{\beta hNm}$ in the denominator in the assertion of the proposition.  We proceed as in the proof of Proposition \ref{prop:expect}. The main difference is that $\mathbb ET(\sigma)$ has an supplementary factor $e^{\beta h |\sigma|}$, which results in an extra factor $e^{\beta h k}$ in  \eqref{eq:first step exp}. The remaining notation is the same as in the proof of Proposition \ref{prop:expect}. Again we use  $\frac kN= m+\frac{c_k}{\sqrt N}$ with $|c_k|\le \kappa_N$, i.e.
\begin{equation*}
e^{\beta h k} = \eee^{\beta h Nm}e^{\beta h \sqrt{N}c_k}.
\end{equation*}
Instead of \eqref{eq:coeff_c_k} we have
\begin{align}
&\sum_{\sigma \in S^1_N}g\left(\frac{|\sigma|-Nm}{\sqrt N}\right)\E T(\sigma) \nonumber\\
& \quad =
(1+o(1))\eee^{\frac{(1-p)\beta^2}{8 p}(m^4+1)-N I(m) + \beta h Nm} 2^N \frac{1}{\sqrt{1-m^2}} \sqrt{\frac {2} {\pi N} } \nonumber\\
&\qquad
\times \sum_{k\in W_{N,m}}
g\left(\frac{k-Nm}{\sqrt N}\right)\eee^{\frac{\beta}{2N}\left(2 N^{3/2}m c_k+c_k^2 N\right)-\frac{(1-p)\beta^2}{8 p}\frac{2m^4N^2+4m^3N^{3/2}c_k+2m^2c_k^2 N}{N^2} + \beta h \sqrt{N}c_k} \nonumber \\
&\qquad  \qquad \times  \eee^{-N I'(m) \frac{c_k}{\sqrt N}
- I''(m)\frac{c_k^2}2}. \nonumber
\end{align}
The linear term in $c_k$ in the exponent is
\begin{align*}
&c_k \left(\beta m \sqrt N -\frac{\beta^2 (1-p)m^3}{2p \sqrt N}-\sqrt N I'(m) + \beta h \sqrt{N} \right)\\
&=c_k \left(\beta m \sqrt N -\sqrt N I'(m) + \beta h \sqrt{N} \right) + o(1)

\\
& =o(1),
\end{align*}
where the first equality follows  from the definition of $\kappa_N$ and $Np^3 \to \infty$ and the second equality follows from
$
m^+(\beta, h)= \tanh(\beta (m^+ (\beta, h)+h)),$
which is the defining relation for $m^+(\beta, h)$.
The rest of the proof for typical spin configurations can  remain unchanged, except for the additional factor $e^{\beta h Nm}$, which now appears in the denominator in the assertion of the proposition.
For atypical spin configurations we get again an extra factor $e^{\beta h k}$ in the expansion of $\mathbb ET(\sigma)$, i.e.~instead of \eqref{eq:atyp_spin_2} we arrive at
\begin{align*}
&\sum_{\sigma \in S^{1c}_{N}}g\left(\frac{|\sigma|-Nm}{\sqrt N}\right)\E T(\sigma) \nonumber\\
\le& C ||g||_{\infty} \eee^{\frac{(1-p)\beta^2}{8 p}(m^4+1)}2^N \sum_{k:\atop
k^2 \in W_{N,m,0}^c}
\eee^{\frac{\beta}{2N}(k^2-m^2N^2)-\frac{m^4(1-p)\beta^2}{4 p}-\frac{2c_k (1-p)\beta^2}{4 p \sqrt N}m^3-NI(\frac kN)} \eee^{\beta h k}.
\end{align*}
The function $\frac{k}{N} \mapsto \frac{\beta}{2N}k^2-NI(\frac kN)+ \beta h k$ attains its maximum in $m^{+}(\beta, h)$
and hence
\begin{equation*}
\frac{\beta}{2N}k^2-NI(\frac kN) + \beta h k\le N\left(\frac{\beta}{2}m^2-I(m ) + \beta h m \right)-k_1 c_k^2,
\end{equation*}
for some $k_1>0$. The rest of the proof is completely analogous to the case $h =0$.
\end{proof}
In the expansion of $\E (T(\sigma)T(\tau))$ we get two extra terms, which leads immediately  to the following analogue of Lemma \ref{CovT}.
\begin{lemma}
For $h>0$, $\beta>0$, all $p=p(N)$ such that $p^3 N \to \infty$ and all $\sigma, \tau \in \{-1,+1\}^N$ we have
\begin{equation*}
\E (T(\sigma)T(\tau))= \exp(N^2 b_0 + b_1 |\sigma|^2 + b_2 |\tau|^2 + b_{12} |\sigma \tau|^2 + h \beta (|\sigma| + |\tau|))
\end{equation*}
with coefficients $b_0, b_1, b_2, b_{12}$ given in Lemma \ref{CovT}.
\end{lemma}
The proof of
\begin{equation*}
\V(\tilde Z_N(\beta, g))=o\left(\E^2[\tilde Z_N(\beta, g)]\right),
\end{equation*}
i.e.~of the result corresponding to Proposition \ref{prop:var}, as well as the rest of the proof is completely analogous to the case $h=0$, where in \eqref{eq:proof_th_1_1} $\xi$ is now a  normally distributed random variable with expectation $0$ and variance $\sigma^2(\beta,h)$. In particular, it follows that
\begin{equation}\label{eq:111}
\frac{\tilde Z_N(\beta, g)}{\E \tilde Z_N(\beta, g)} \to 1,
\qquad
\frac{\tilde Z_N(\beta)}{\E \tilde Z_N(\beta)} \to 1,
\end{equation}
in probability.

\section{Proof of Theorem \ref{theo:fluc_part}}
We give the proof for the two cases $h=0$, $\beta>1$ and $h>0$, $\beta>0$ simultaneously. The main ingredient for the fluctuations of $Z_N(\beta)$ is the relation (see \eqref{eq:relation_Z_tilde_and_Z})
\begin{equation*}
Z_N(\beta)= \tilde Z_N(\beta) \exp \left(  \gamma m^2 \sum_{i,j=1}^N \varepsilon_{i,j}\right)
\end{equation*}
together with
\begin{equation*}
\frac{ \tilde Z_N(\beta)}{\E \tilde Z_N(\beta)} \to 1 \quad \text{in probability};
\end{equation*}
see \eqref{eq:in_probab_noplus} for the case $h=0$, $\beta>1$, and~\eqref{eq:111} for the case $h>0$, $\beta>0$.
The term $\exp \left(  \gamma m^2 \sum_{i,j=1}^N \varepsilon_{i,j}\right)$ can easily be treated by the Central Limit Theorem.
\begin{proof}[Proof of Theorem \ref{theo:fluc_part}]
We have
\begin{equation*}
\log \frac{Z_N(\beta)}{\E \tilde Z_N(\beta)} = \log \frac{ \tilde Z_N(\beta)}{\E \tilde Z_N(\beta)}
 +  \gamma m^2 \sum_{i,j=1}^N \varepsilon_{i,j}.
\end{equation*}
Note that $\gamma  \sum_{i,j=1}^N \varepsilon_{i,j}$ is a sum of independent   random variables with
\begin{equation*}
\mathbb E\left( \gamma  \sum_{i,j=1}^N \varepsilon_{i,j} \right) = \gamma  p N^2= \frac{\beta}{2} N
\end{equation*}
and
\begin{equation*}
\mathbb V\left( \gamma  \sum_{i,j=1}^N \varepsilon_{i,j} \right)= \gamma^2 N^2 p(1-p) = \frac{\beta^2}{4p}(1-p).
\end{equation*}
Hence, by the Central Limit Theorem (which applies since $N^2p(1-p)\to\infty$) we have for $N \to \infty$
\begin{equation*}
\frac{\gamma m^2 \sum_{i,j=1}^N \varepsilon_{i,j}- m^2 \frac{\beta}{2} N}{\sqrt{m^4 \frac{\beta^2}{4p}(1-p)}} \to  \mathfrak N_{0,1}
\end{equation*}
in distribution. It follows that
$$
\frac{\log \frac{Z_N(\beta)}{\E \tilde Z_N(\beta)}- m^2 \frac{\beta}{2} N }{\sqrt{m^4 \frac{\beta^2}{4p}(1-p)}}
= \frac{\log \frac{ \tilde Z_N(\beta)}{\E \tilde Z_N(\beta)}}{\sqrt{m^4 \frac{\beta^2}{4p}(1-p)}}
+
\frac{\gamma m^2 \sum_{i,j=1}^N \varepsilon_{i,j}- m^2 \frac{\beta}{2} N}{\sqrt{m^4 \frac{\beta^2}{4p}(1-p)}}.
$$
The numerator of the first summand converges to $0$ in probability, while the denominator is bounded away from $0$ because $p$ is bounded away from $1$. It follows that the first term converges to $0$ in probability and hence
$$
\frac{\log \frac{Z_N(\beta)}{\E \tilde Z_N(\beta)}- m^2 \frac{\beta}{2} N }{\sqrt{m^4 \frac{\beta^2}{4p}(1-p)}} \to \mathfrak N_{0,1}
$$
in distribution. This completes the proof.
\end{proof}

\bibliographystyle{abbrv}
\bibliography{LiteraturDatenbank}
\end{document}